\documentclass[twoside,a4paper,reqno,11pt]{amsart}
\usepackage{amsfonts, amsbsy, amsmath, amssymb, latexsym}
\usepackage{mathrsfs,array}
\usepackage[top=30mm,right=30mm,bottom=30mm,left=30mm]{geometry}
\usepackage{stmaryrd}
\usepackage{bm}
\usepackage{xcolor}
\usepackage{tikz-cd}
\usepackage{url}
\usepackage{hyperref}

\makeatletter
\@namedef{subjclassname@2020}{\textup{2020} Mathematics Subject Classification}
\makeatother

\newcommand{\ovl}{\overline}

\newcommand{\Z}{\mathbb{Z}}
\newcommand{\N}{\mathbb{N}}

 \renewcommand{\to}{\rightarrow}

\DeclareMathOperator{\Gal}{Gal}

\DeclareMathOperator{\GL}{GL}
\DeclareMathOperator{\PGL}{PGL}
\DeclareMathOperator{\SL}{SL}
\DeclareMathOperator{\Lim}{lim}
\DeclareMathOperator{\red}{red}
\DeclareMathOperator{\Hom}{Hom}

\DeclareMathOperator{\Spec}{Spec}
\DeclareMathOperator{\id}{id}
\DeclareMathOperator{\op}{op}
\DeclareMathOperator{\Grp}{Grp}
\DeclareMathOperator{\Lie}{Lie}

\DeclareMathOperator{\Alg}{Alg}

\DeclareMathOperator{\AlgGrp}{AlgGrp}
\DeclareMathOperator{\Sm}{sm}
\DeclareMathOperator{\sep}{s}
\DeclareMathOperator{\Char}{char}
\DeclareMathOperator{\rank}{rank}
\newcommand{\Ga}{{\mathbb G}_a}
\newcommand{\Gm}{{\mathbb G}_m}

\makeatletter
\newcommand{\imod}[1]{\allowbreak\mkern4mu({\operator@font mod}\,\,#1)}
\makeatother

\newtheorem{thm}{Theorem}[section]
\newtheorem{prop}[thm]{Proposition}
\newtheorem{lem}[thm]{Lemma}
\newtheorem{cor}[thm]{Corollary}

\theoremstyle{definition}
\newtheorem{rem}[thm]{Remark}
\newtheorem{defn}[thm]{Definition}
\newtheorem{example}[thm]{Example}
\newtheorem{question}[thm]{Question}

\begin{document}

\title{The subgroup structure of pseudo-reductive groups}

\author[M.\  Bate]{Michael Bate}
\address
{Department of Mathematics,
University of York,
York YO10 5DD,
UK}
\email{michael.bate@york.ac.uk}
\author[B.\  Martin]{Ben Martin}
\address
{Institute of Mathematics,
University of Aberdeen,
Aberdeen AB24 3FX,
UK}
\email{b.martin@abdn.ac.uk}
\author[G. R\"ohrle]{Gerhard R\"ohrle}
\address
{Fakult\"at f\"ur Mathematik,
Ruhr-Universit\"at Bochum,
D-44780 Bochum, Germany}
\email{gerhard.roehrle@rub.de}
\author[D.\ Sercombe]{Damian Sercombe}
\address
{Fakult\"at f\"ur Mathematik,
Ruhr-Universit\"at Bochum,
D-44780 Bochum, Germany}
\email{damian.sercombe@rub.de}

\begin{abstract} 
Let $k$ be a field. We investigate the relationship between subgroups of a pseudo-reductive $k$-group $G$ and its maximal reductive quotient $G'$, with applications to the subgroup structure of $G$. Let $k'/k$ be the minimal field of definition for the geometric unipotent radical of $G$, and let $\pi':G_{k'} \to G'$ be the quotient map. We first characterise those smooth subgroups $H$ of $G$ for which $\pi'(H_{k'})=G'$. We next consider the following questions: given a subgroup $H'$ of $G'$, does there exist a subgroup $H$ of $G$ such that $\pi'(H_{k'})=H'$, and if $H'$ is smooth can we find such a $H$ that is smooth? We find sufficient conditions for a positive answer to these questions. In general there are various obstructions to the existence of such a subgroup $H$, which we illustrate with several examples. Finally, we apply these results to relate the maximal smooth subgroups of $G$ with those of $G'$.
\end{abstract}

\maketitle

\section{Introduction}
\label{sec:intro}

 A fundamental problem in group theory is to describe the subgroup structure of a given group. In this paper we study the subgroup structure of pseudo-reductive groups over a field $k$.
 Pseudo-reductive groups have received considerable attention over the last decade and a half \cite{BRSS}, \cite{CGP}, \cite{CP}, \cite{CP1}.
 
  Let $G$ be a pseudo-reductive $k$-group. Let $\ovl{k}$ be an algebraic closure of $k$, and let $k^{\sep}$ be the separable closure of $k$ in $\ovl{k}$. Let $k'/k$ be the minimal field of definition for the geometric unipotent radical $\mathscr{R}_u(G_{\ovl{k}})$. Then $\mathscr{R}_u(G_{\ovl{k}})$ descends to $k'$, and $G':= G_{k'}/\mathscr{R}_u(G_{k'})$ is a reductive $k'$-group.  We denote the canonical projection by $\pi':G_{k'} \to G'$, and the associated map under the adjunction of extension of scalars and Weil restriction by $i_G:G \to R_{k'/k}(G')$.   Given a (not necessarily smooth) subgroup $H$ of $G$, we obtain a subgroup $H':= \pi'(H_{k'})$ of the reductive $k'$-group $G'$.  This gives a mapping ${\mathcal S}$ from the set of subgroups of $G$ to the set of subgroups of $G'$ defined by $H\mapsto \pi'(H_{k'})$.
  
  We investigate the relationship between subgroups $H$ of $G$ and subgroups $H'$ of $G'$.  Our motivation is that it is easier to understand subgroups of a reductive group than those of a pseudo-reductive group; for example, a great deal is known about the reductive subgroups of a simple exceptional group when $k= \ovl{k}$ \cite{LS2}.  The mapping ${\mathcal S}$ allows us to pass from subgroups of $G$ to subgroups of $G'$.  To reverse this process, we need the following definition.
  
 \begin{defn}\label{defn:levitation} 
 Let $H'$ be a subgroup of $G'$. A subgroup $H$ of $G$ satisfying $\pi'(H_{k'})=H'$ is called a \emph{levitation} of $H'$ in $G$; we say that $H'$ \emph{levitates} to $H$ in $G$. If such a subgroup $H$ exists then we say that $H'$ \emph{levitates} or {\em is levitating} in $G$. If, in addition, $H$ (and hence $H'$) is smooth then $H$ is called a \emph{smooth levitation} of $H'$ in $G$, and we say that $H'$ \emph{smoothly levitates} or is \emph{smoothly levitating} (to $H$) in $G$.
\end{defn}

Usually ${\mathcal S}$ is not injective: that is, a given subgroup of $G'$ can admit more than one levitation.  We show, however, that if $H'$ levitates then $H'$ has a unique largest levitation, and if $H'$ levitates smoothly then $H'$ has a unique largest smooth levitation (Proposition~\ref{largestlevitation}).  An obvious question is whether ${\mathcal S}$ is surjective: that is, does every subgroup of $G'$ levitate?  This is more delicate.  We prove some results showing that certain subgroups of $G'$ levitate under certain additional assumptions, and we give a collection of examples that exhibit various obstructions to levitation.  Some of these subtleties can occur only when the root system of $G$ is non-reduced, so they are restricted to characteristic 2.

We start with the case $H'= G'$.  Here is our first main result.  Recall that a subgroup $L$ of $G$ is called a \emph{Levi subgroup} if it satisfies $\smash{G_{\ovl{k}}=\mathscr{R}_u(G_{\ovl{k}}) \rtimes L_{\ovl{k}}}$ (scheme-theoretically). We introduce a slightly weaker notion. A smooth subgroup $L$ of $G$ is an \emph{almost Levi subgroup} if $\smash{G(\ovl{k})=\mathscr{R}_u(G(\ovl{k})) \rtimes L(\ovl{k})}$ (as abstract groups). 
Given an almost Levi subgroup $L$ of $G$, if $L$ is not a Levi subgroup then $k$ has characteristic 2 and the root system of $G_{k^{\sep}}$ is non-reduced.

\begin{thm}
\label{thm:levitating}
Let $G$ be a pseudo-reductive $k$-group. Let $H$ be a smooth subgroup of $G$, and denote $H':=\pi'(H_{k'})$.
\begin{itemize}
\item[(i)] Suppose $k=k^{\sep}$. Then $H'=G'$ if and only if $H$ contains an almost Levi subgroup of $G$.

\item[(ii)] Suppose the root system of $G_{k^{\sep}}$ is reduced. If $H'=G'$ then the restriction map $\pi'|_{H_{k'}}:H_{k'} \to G'$ is smooth.
\end{itemize}
\end{thm}


Next we turn our attention to tori.

\begin{thm}
\label{thm:levitating_crit}
 Let $G$ be a pseudo-reductive $k$-group. Suppose the root system of $G_{k^{\sep}}$ is reduced. The following are equivalent.
 \begin{itemize}
  \item[(a)] $i_G(G)$ contains $\mathscr{D}(R_{k'/k}(G'))$.
  \item[(b)] $i_G(G)$ contains every torus of $R_{k'/k}(G')$.
  \item[(c)] For every torus $S'$ of $G'$, there exists a subgroup $S$ of $G$ such that $\pi'(S_{k'})=S'$. 
  \item[(d)] For every torus $S'$ of $G'$, there exists a torus $S$ of $G$ such that $\pi'(S_{k'})=S'$. 
  \item[(e)] For every maximal torus $T'$ of $G'$, there exists a maximal torus $T$ of $G$ such that $\pi'(T_{k'})=T'$. 
  \item[(f)] $G$ is standard, and the minimal field of definition for the geometric unipotent radical of each pseudo-simple component of $G$ equals $k'/k$.
 \end{itemize}
\end{thm}

\noindent We find an example where a maximal torus of $G'$ does not levitate (Example~\ref{ex:non_levitating}), and an example where a maximal torus of $T$ levitates but does not levitate smoothly (Example~\ref{BC1torusnotlift}). 

Things are murkier when we move away from the case when $H'$ is a torus, even when the root system of $G_{k^{\sep}}$ is reduced and conditions (a)--(e) of Theorem~\ref{thm:levitating_crit} hold. We find an example where a smooth wound unipotent subgroup of $G'$ does not levitate at all (Example~\ref{ex:smooth_nonlev}), and an example where $G$ is of the form $R_{k'/k}(G')$ for some reductive $G'$ and there is a non-smooth subgroup of $G'$ that does not levitate (Example~\ref{ex:nonsmooth_nonlev}). 

On the other hand, the following result shows that the relationship between subgroups of $G$ and $G'$ is quite well-behaved if we restrict ourselves to regular smooth subgroups.  (Recall that a subgroup of $G$ is called \emph{regular} if it is normalised by some maximal torus of $G$.)

\begin{thm}\label{thm:levitating_regsubgroups} Let $G$ be a pseudo-reductive $k$-group. Let $H'$ be a smooth subgroup of $G'$. Suppose there exists a maximal torus $T$ of $G$ such that $\pi'(T_{k'})$ normalises $H'$.
\begin{itemize}
\item[(i)] There exists a largest smooth subgroup $H$ of $G$ such that $\pi'(H_{k'})=H'$. Moreover $H$ is normalised by $T$. 

\item[(ii)] $(H')^{\circ}$ is a maximal torus (resp. root group, parabolic subgroup) of $G'$ if and only if $H^{\circ}$ is a Cartan subgroup (resp. root group, pseudo-parabolic subgroup) of $G$.
\end{itemize}
\end{thm}

We finish the paper by considering maximal smooth subgroups of $G$ and $G'$.  It is not true in general that ${\mathcal S}$ takes maximal smooth subgroups of $G$ to maximal smooth subgroups of $G'$, or that a smooth levitation of a maximal smooth subgroup of $G'$ must be a maximal smooth subgroup of $G$: for instance, in Example~\ref{BC1torusnotlift} we show that a maximal smooth subgroup of $G'$ need not have a smooth levitation, while in Example~\ref{ex:non_maximal_levitation} we find a maximal smooth subgroup $H'$ of $G'$ such that $H'$ levitates smoothly but the largest smooth levitation of $H'$ is not maximal in $G$.  We do have the following result, which gives a step towards reducing the problem of describing maximal smooth subgroups from the case of an arbitrary pseudo-reductive group to the case of a reductive group.

\begin{thm}
\label{thm:nicetheorem}
 Let $G$ be a pseudo-reductive $k$-group.
\begin{itemize}
\item[(i)] Let $H$ be a maximal smooth subgroup of $G$. Suppose at least one of the following conditions hold: 
\begin{itemize} 
	\item[(a)] $G=R_{k'/k}(G')$; or 
	\item[(b)] $H$ is a regular subgroup of $G$. 
\end{itemize} 
Then either $\pi'(H_{k'})$ is a maximal smooth subgroup of $G'$, or $H_{k^{\sep}}$ contains an almost Levi subgroup of $G_{k^{\sep}}$.

\item[(ii)] Let $H'$ be a maximal smooth subgroup of $G'$. Suppose $H'$ smoothly levitates in $G$, and let $M$ be a smooth subgroup of $G$ that properly contains the largest smooth levitation of $H'$. Then $M_{k^{\sep}}$ contains an almost Levi subgroup of $G_{k^{\sep}}$.
\end{itemize}
\end{thm}

The paper is set out as follows.  We cover some preliminary material in Section~\ref{sec:prelim}, and in Section~\ref{sec:leviovergroups} we deal with the case $H'= G'$ and we prove Theorem~\ref{thm:levitating}.  The proof of Theorem~\ref{thm:levitating_crit} and various counter-examples appear in Section~\ref{sec:levitatingsection}.  In Section~\ref{sec:maximal} we consider maximal subgroups and we prove Theorem~\ref{thm:nicetheorem}.

\section{Preliminaries}
\label{sec:prelim}



Let $k$ be a field. Let $\overline{k}$ be an algebraic closure of $k$, and let $k^{\sep}$ be the separable closure of $k$ in $\overline{k}$. By a \textit{$k$-group} we mean a group scheme over $k$, and by an \textit{algebraic $k$-group} we mean a group scheme of finite type over $k$. We denote by $\AlgGrp\hspace{-0.5mm}/k$ the category of algebraic $k$-groups and homomorphisms between them. We do not assume that algebraic $k$-groups are smooth. By a \textit{subgroup} $H$ of an algebraic $k$-group $G$ we mean a locally closed subgroup scheme; note that $H$ is automatically closed by \cite[047T]{Stacks} and is of finite type by \cite[01T5, 01T3]{Stacks}. In this section we allow for not-necessarily-affine algebraic $k$-groups, whilst for the remainder of the paper we are only concerned with affine groups. We denote by $\alpha_p$ and $\mu_p$ the first Frobenius kernels of the additive group $\Ga$ and the multiplicative group $\Gm$, respectively.

\subsection{Some structure theory}\label{structuretheory}

Let $G$ be an algebraic $k$-group. We denote the centre of $G$ by $Z(G)$. That is, for $A$ a $k$-algebra, $Z(A)$ represents the functor $$A \mapsto \{g \in G(A) \hspace{0.5mm}|\hspace{0.5mm} \textnormal{$g^{-1}xg=x$, for all  $A$-algebras $A'$ and all $x \in G(A')$}\}.$$ If $k'/k$ is a field extension then $G_{k'}$ denotes the algebraic $k'$-group obtained from $G$ by extension of scalars. Let $G_{\red}$ denote the reduction of $G$ (i.e. the unique reduced closed $k$-subscheme of $G$ with the same underlying topological space). Let $G^{\Sm}$ denote the subgroup of $G$ that is generated by $G(k^{\sep})$ (more precisely take the subgroup of $G_{k^{\sep}}$ that is generated by its $k^{\sep}$-points, observe that it is stable under the action of $\Gal(k^{\sep}/k)$, then take its $k$-descent). Then $G^{\Sm}$ is the largest smooth subgroup of $G$. Note that $G^{\Sm}$ is contained in $G_{\red}$ but it is not necessarily normal in $G$ (see \cite[Ex.\ 2.35]{Mi}, for example).
 

Now assume that $G$ is a smooth connected affine $k$-group. The \emph{unipotent radical} $\mathscr{R}_u(G)$ of $G$ is the largest smooth connected unipotent normal subgroup of $G$. If $\mathscr{R}_u(G)$ is trivial then $G$ is \emph{pseudo-reductive}.  For a comprehensive account of pseudo-reductive groups and their properties, see \cite{CGP}. 

The notions of a root system and root groups are particularly important; see \cite[\S 2.3]{CGP}, where the necessary constructions are carried out for an arbitrary smooth connected affine $k$-group. The root system of a reductive group is always reduced, but a pseudo-reductive group can have a non-reduced root system --- see \cite[9.3]{CGP} --- and this leads to complications (see, e.g., the paragraph before Lemma~\ref{nonliftingtorus}). To clarify, we fix some maximal torus $T$ of $G$. By a \textit{root} (resp. \textit{root group}) of $G$ we mean a root (resp. root group) of $G_{k^{\sep}}$ with respect to $T_{k^{\sep}}$ which is stable under the action of $\Gal(k^{\sep}/k)$. Note that we are referring to the ``absolute" notions of root systems and root groups, which are not to be confused with the ``relative" notions in the sense of \cite[C.2.13]{CGP}.

The \emph{split unipotent radical} $\mathscr{R}_{us}(G)$ of $G$ is the largest smooth connected split unipotent normal subgroup of $G$. If $\mathscr{R}_{us}(G)$ is trivial then $G$ is \emph{quasi-reductive}. The group $G$ is called \emph{pseudo-split} if it contains a split maximal torus, and \emph{pseudo-simple} if it is non-commutative and has no nontrivial smooth connected proper normal subgroup.  We say that $G$ is \emph{absolutely pseudo-simple} if $G_{k^{\sep}}$ is pseudo-simple. There are natural notions of pseudo-semisimple groups and pseudo-simple components.

The geometric unipotent radical of $G$ is the group $\mathscr{R}_u(G_{\ovl{k}})$. If this group is trivial then we say that $G$ is \emph{reductive}.  Note our convention is that reductive and pseudo-reductive groups are connected.
 
  A subgroup $L$ of $G$ satisfying $\smash{G_{\ovl{k}}=\mathscr{R}_u(G_{\ovl{k}}) \rtimes L_{\ovl{k}}}$ is called a \emph{Levi subgroup} of $G$; Levi subgroups are reductive.  If $G$ is pseudo-split  and pseudo-reductive, or if $k$ has characteristic $0$, then $G$ admits a Levi subgroup (see \cite[3.4.6]{CGP} and \cite[25.49]{Mi}). However in general $G$ need not contain a Levi subgroup: this failure of existence may occur over any algebraically closed field of positive characteristic (see \cite[A.6]{CGP}), or if $G$ is pseudo-reductive but not pseudo-split (refer to \cite[7.2.2]{CGP}). We introduce a related notion.
\begin{defn}\label{defn:almostLevi}
A smooth subgroup $L$ of $G$ is an \emph{almost Levi subgroup} if $G(\ovl{k})=\mathscr{R}_u(G(\ovl{k})) \rtimes L(\ovl{k})$ (as abstract groups).
\end{defn}

 Fix a cocharacter $\lambda:\mathbb{G}_m \to G$. For any $k$-algebra $A$ and $g \in G(A)$, we say that $\Lim_{t \to 0} \lambda(t)g\lambda(t)^{-1}$ \textit{exists} if the $A$-scheme map $\Gm \to G_A$ defined by $t \mapsto \lambda(t)g\lambda(t)^{-1}$ extends (necessarily uniquely) to an $A$-scheme map $\mathbb{A}^1 \to G_A$. 
The functor $$A \mapsto \{g\in G(A) \mid \Lim_{t \to 0} \lambda(t)g\lambda(t)^{-1} \textnormal{ exists}\}$$ is representable as a subgroup (scheme) of $G$, we denote it by $P_G(\lambda)$. A subgroup of $G$ is called \emph{pseudo-parabolic} if it is of form $P_G(\lambda)\mathscr{R}_u(G)$ for some cocharacter $\lambda:\mathbb{G}_m \to G$; so if $G$ is pseudo-reductive then each $P_G(\lambda)$ is pseudo-parabolic. Every pseudo-parabolic subgroup of $G$ is smooth and connected. A pseudo-parabolic subgroup of $G$ is called \emph{maximal pseudo-parabolic} if it is maximal amongst all proper pseudo-parabolic subgroups of $G$.

 Now specialise to the case where $G$ is a pseudo-split pseudo-reductive $k$-group. Let $T$ be a maximal torus of $G$ and let $\Phi:=\Phi(G,T)$ be the root system of $G$ with respect to $T$. 
 For any $\alpha \in \Phi$ there exists a unique $T$-stable smooth connected subgroup $U_{\alpha}$ of $G$ such that $\Lie(U_{\alpha})$ is the span of the root spaces in $\Lie(G)$ which correspond to multiples $n\alpha$ for $n \in \N$. We call $U_{\alpha}$ the $T$-\emph{root group} of $G$ associated to $\alpha$. 
 Without the pseudo-splitness assumption, a subgroup $H$ of $G$ is called a \emph{root group} if $H_{k^{\sep}}$ is a root group of $G_{k^{\sep}}$.

\subsection{Extension of scalars and Weil restriction}\label{theadjunction} 
Let $k'/k$ be a finite field extension.
 Let $G'$ be an algebraic $k'$-group. The \emph{Weil restriction} of $G'$ by $k'/k$ is the functor $(\Alg\hspace{-0.5mm}/k)^{\op} \to \Grp$ given on objects by $A \mapsto \Hom_{k'}(\Spec A\otimes_k k',G') =:G'(A\otimes_k k')$, where $\Alg\hspace{-0.5mm}/k$ refers to the category of finite-dimensional commutative $k$-algebras, and $\Grp$ the category of (abstract) groups. This functor is representable as an algebraic $k$-group (see \cite[\S 7.6/4]{BLR} or \cite[A.5.1]{CGP}); let us denote it by $R_{k'/k}(G')$. We have an induced functor \begin{equation}\label{Weilrest}R_{k'/k}(-):\AlgGrp\hspace{-0.5mm}/k' \to \AlgGrp\hspace{-0.5mm}/k\end{equation} called \emph{Weil restriction} by $k'/k$.

Many useful properties of algebraic groups over a field, in particular affineness and smoothness, are preserved by both the extension of scalars functor and the Weil restriction functor. Both functors also preserve the property of being a monomorphism.

 It is well-known that the Weil restriction functor is right adjoint to the extension of scalars functor (see for instance \cite[Lem.\ 7.6/1]{BLR}). It follows that Weil restriction preserves pullbacks: in particular, it preserves preimages. We make explicit some other properties of this adjunction, given that we will use them repeatedly.


For an algebraic $k$-group $G$ and an algebraic $k'$-group $G'$ there exists a natural bijection \begin{equation}\label{adjunctbij}\Hom_k(G,R_{k'/k}(G')) \xrightarrow{\sim} \Hom_{k'}(G_{k'},G').\end{equation} 
Let \begin{equation}\label{unitadjunctioncomp} j_G:G \to R_{k'/k}(G_{k'}) \end{equation} be the component at $G$ of the unit of this adjunction: that is, $j_G$ is the map corresponding to $G_{k'}\stackrel{\rm id}{\to} G_{k'}$ via (\ref{adjunctbij}). Given any homomorphism of $k$-groups $f:H\to G$, the following diagram commutes: \begin{equation}\label{functorialityofjG}\begin{tikzpicture}[every node/.style={midway},baseline=(current  bounding  box.center)]
\matrix[column sep={8.5em,between origins},
        row sep={3.7em}] at (0,0)
{ \node(A) {$H$}  ; & \node(B) {$R_{k'/k}(H_{k'})$}; \\
  \node(C) {$G$}; & \node(D) {$R_{k'/k}(G_{k'})$}; \\};
\draw[<-] (C) -- (A) node[anchor=east]  {$f$};
\draw[->] (B) -- (D) node[anchor=west]  {$R_{k'/k}(f_{k'})$};
\draw[->] (A)   -- (B) node[anchor=south] {$j_H$};
\draw[->] (C)   -- (D) node[anchor=south] {$j_G$};
\end{tikzpicture}\hspace{-1.3cm}\end{equation} Given a $k'$-homomorphism $\psi':G_{k'} \to G'$, let $\psi:G \to R_{k'/k}(G')$ be the $k$-homomorphism associated to $\psi'$ under the bijection in $(\ref{adjunctbij})$. The universal property of $j_G$ says that the following triangle commutes: \begin{equation}\label{universalpropertyofjG}\begin{tikzcd}[row sep={4.5em,between origins},column sep={6.8em,between origins},]
G \arrow[r, "j_G"] \arrow[dr, "\psi"]
& R_{k'/k}(G_{k'}) \arrow[d, "R_{k'/k}(\psi')"]\\
& R_{k'/k}(G') \end{tikzcd}\end{equation} 
Now let \begin{equation}\label{counitadjunctioncomp} q_{G'}:R_{k'/k}(G')_{k'} \to G' \end{equation} be the component at $G'$ of the counit of this adjunction: that is, $q_{G'}$ is the map corresponding to $R_{k'/k}(G)\stackrel{\rm id}{\to} R_{k'/k}(G)$ via (\ref{adjunctbij}).  If $G'$ is smooth then \cite[A.5.11(1)]{CGP} tells us that $q_{G'}:G_{k'} \to G'$ is smooth and surjective.  Given any homomorphism of algebraic $k'$-groups $f':H' \to G'$, the following diagram commutes: \begin{equation}\label{functorialityofqG'}\hspace{-1.5cm}\begin{tikzpicture}[every node/.style={midway},baseline=(current  bounding  box.center)]
\matrix[column sep={8.5em,between origins},
        row sep={3.7em}] at (0,0)
{ \node(A) {$R_{k'/k}(H')_{k'}$}  ; & \node(B) {$H'$}; \\
  \node(C) {$R_{k'/k}(G')_{k'}$}; & \node(D) {$G'$}; \\};
\draw[<-] (C) -- (A) node[anchor=east]  {$R_{k'/k}(f')_{k'}$};
\draw[->] (B) -- (D) node[anchor=west]  {$f'$};
\draw[->] (A)   -- (B) node[anchor=south] {$q_{H'}$};
\draw[->] (C)   -- (D) node[anchor=south] {$q_{G'}$};
\end{tikzpicture}\end{equation} Given a $k$-homomorphism $\phi:G \to R_{k'/k}(G')$, let $\phi':G_{k'} \to G'$ be the $k'$-homomorphism associated to $\phi$ under the bijection in $(\ref{adjunctbij})$. The universal property of $q_{G'}$ says that the following triangle commutes: \begin{equation}\label{universalpropertyofqG'}\begin{tikzcd}[row sep={4.5em,between origins},column sep={6.8em,between origins},]
G_{k'} \arrow[r, "\phi_{k'}"] \arrow[dr, "\phi'"]
& R_{k'/k}(G')_{k'} \arrow[d, "q_{G'}"]\\
& G' \end{tikzcd}\end{equation}

The following two equations are respectively known as the first and second counit-unit equations of the adjunction; they follow from the universal properties of $j_G$ and $q_{G'}$. If $G_{k'}=G'$ then \begin{equation}\label{firstcounitunitequation} q_{G'} \circ (j_G)_{k'} = \id_{G'}.\end{equation} If $G=R_{k'/k}(G')$ then \begin{equation}\label{secondcounitunitequation} R_{k'/k}(q_{G'}) \circ j_G = \id_G.\end{equation}

 In this paper we are mostly concerned with the following special situation. Let $G$ be a pseudo-reductive $k$-group, and let $k'/k$ be the minimal field of definition for its geometric unipotent radical $\mathscr{R}_u(\smash{G_{\ovl{k}}})$. Let $\pi':G_{k'} \to G_{k'}/\mathscr{R}_u(G_{k'})=:G'$ be the natural projection. We define a map \begin{equation}\label{iGmap} i_G:G \to R_{k'/k}(G')\end{equation} to be the $k$-homomorphism associated to $\pi'$ under the bijection in $(\ref{adjunctbij})$.  Note that $\ker i_G$ is unipotent, so $(\ker i_G)^{\Sm}$ is \'etale (see \cite[1.6 and Lem.\ 1.2.1]{CGP}), and $\ker i_G$ is central if the root system of $G_{k^{\sep}}$ is reduced \cite[2.3.4]{CP}.  Applying (\ref{universalpropertyofqG'}) with $\phi= i_G$ gives
 \begin{equation}
 \label{pi'_eqn}
  q_{G'}\circ (i_G)_{k'}= \pi'.
 \end{equation}
 
 In the special case when $G= R_{k'/k}(G')$, $i_G$ gives an isomorphism from $G$ onto itself \cite[Thm.\ 1.6.2(2)]{CGP}. More generally, if $G$ is a subgroup of $R_{k'/k}(G')$ then we may regard $i_G$ as the inclusion map.
 
 We wish to study the relationship between subgroups of $G'$ and subgroups of $G$, which motivates the following terminology.

\begin{defn}
Suppose $G$ is pseudo-reductive and $\pi':G_{k'} \to G'$ is the map just defined. 
Let $H'$ be a subgroup of $G'$. A subgroup $H$ of $G$ satisfying $\pi'(H_{k'})=H'$ is called a \emph{levitation} of $H'$ in $G$; we say that $H'$ \emph{levitates} to $H$ in $G$. If such a subgroup $H$ exists then we say that $H'$ \emph{levitates} or {\em is levitating} in $G$. If, in addition, $H$ (and hence $H'$) is smooth then $H$ is called a \emph{smooth levitation} of $H'$ in $G$, and we say that $H'$ \emph{smoothly levitates} or is \emph{smoothly levitating} (to $H$) in $G$.
\end{defn}

\subsection{Subgroups of a Weil restriction}\label{subgroupsweilres} 

In this subsection we prove two useful lemmas which relate the subgroup structure of a smooth algebraic $k'$-group $G'$ to that of its Weil restriction $R_{k'/k}(G')$.
Keep the assumption that $k'/k$ is a finite field extension. Let $G'$ be a smooth algebraic $k'$-group (not necessarily affine), and let $G:=R_{k'/k}(G')$. If $G'$ is reductive then $G$ is pseudo-reductive \cite[Prop.\ 1.1.10]{CGP}, if in addition $k'/k$ is purely inseparable then $k'/k$ is the minimal field of definition for $\mathscr{R}_u(G_{\ovl{k}})$ and $i_G\colon G\to G$ is an isomorphism \cite[Thm.\ 1.6.2(2)]{CGP}.
Note that most pseudo-reductive groups are obtained from a group of the form $R_{k'/k}(G')$ via the so-called standard construction \cite[Thm.\ 1.5.1]{CGP}.
 
\begin{lem}\label{adjunctionprop} Let $H'$ be a subgroup of $G'$. There is a canonical inclusion $q_{G'}(R_{k'/k}(H')_{k'}) \subseteq H'$, with equality if $H'$ is smooth. Consequently, if $H'$ is smooth and $R_{k'/k}(H') = R_{k'/k}(G')$ then $H' = G'$. 
\begin{proof} Let $\phi':H' \hookrightarrow G'$ be the inclusion. We may interpret $\smash{R_{k'/k}(\phi')_{k'}}$ as an inclusion. Consider the $k'$-homomorphism $$\nu':=q_{G'}\circ R_{k'/k}(\phi')_{k'}:R_{k'/k}(H')_{k'} \to G'.$$ That is, $\nu'$ is the restriction of $q_{G'}$ to $R_{k'/k}(H')_{k'}$. By the functoriality of $q_{G'}$ (i.e., Diagram (\ref{functorialityofqG'})), we have that $$\phi' \circ q_{H'}=q_{G'}\circ R_{k'/k}(\phi')_{k'}=\nu'.$$ Taking the image of this map $\nu'$ and leaving the inclusion maps as implicit gives us a canonical inclusion $q_{G'}(R_{k'/k}(H')_{k'}) \subseteq H'$.  If $H'$ is smooth then $q_{H'}$ is smooth and surjective, so indeed $q_{G'}(R_{k'/k}(H')_{k'})=H'$. The second assertion then follows immediately.
\end{proof}
\end{lem}

 The inclusion $q_{G'}(R_{k'/k}(H')_{k'}) \subseteq H'$ in Lemma \ref{adjunctionprop} may be strict if $H'$ is not smooth; for an example with reductive $G'$ see Example \ref{ex:nonsmooth_nonlev}.

\begin{lem}\label{adjunctionpropnew} Let $H$ be a subgroup of $G$. Then there are canonical inclusions $H \subseteq R_{k'/k}(q_{G'}(H_{k'})) \subseteq G$.
\begin{proof} Consider the subgroup $H':=q_{G'}(H_{k'})$ of $G'$. Let $\rho':H_{k'} \to H'$ be the $k'$-homomorphism obtained by restricting the domain of $q_{G'}$ to $H_{k'}$, and then restricting the codomain to its image $H'$. Consider the $k$-homomorphism $$\iota:H \to R_{k'/k}(H')$$ associated to $\rho'$ under the bijection in $(\ref{adjunctbij})$.

 Let $\phi:H \hookrightarrow G$ be the inclusion, and let $\phi':H' \hookrightarrow G'$ be the inclusion induced by $\phi$. By definition $\phi' \circ \rho'=q_{G'} \circ \phi_{k'}$. The universal property of $j_H$ (i.e., Diagram $(\ref{universalpropertyofjG})$) says that $\iota=R_{k'/k}(\rho') \circ j_H$. Recall from Equation $(\ref{secondcounitunitequation})$ that $R_{k'/k}(q_{G'}) \circ j_G$ is the identity map on $G$. By the functoriality of $j_G$ (i.e., Diagram $(\ref{functorialityofjG})$), we have that $R_{k'/k}(\phi_{k'}) \circ j_H = j_G \circ \phi$. Combining all of the above gives us $$R_{k'/k}(\phi') \circ \iota=\phi.$$ Since $\phi$ is a monomorphism, this is also true of $\iota$. Leaving the inclusions as implicit, we have shown that $H \subseteq R_{k'/k}(H') \subseteq G$.
\end{proof}
\end{lem}

\section{Overgroups of almost Levi subgroups}
\label{sec:leviovergroups}

 In this section we prove Theorem~\ref{thm:levitating}, and then we illustrate it with some examples. We will need the following easy lemmas about reductive groups.

\begin{lem}\label{unipmap} 
Let $G$ be a reductive $k$-group. Let $N$ be a unipotent normal subgroup of $G$. Then $N$ is infinitesimal.
\end{lem}
\begin{proof} 
Since $G$ is reductive, the subgroup of $\smash{G_{\overline{k}}}$ generated by $N(\overline{k})$ is finite, 
and so it is central in $G$. 
But $Z(G)$ is of multiplicative type as $G$ is reductive, hence $N(\overline{k})$ is trivial. That is, $N$ is infinitesimal.
\end{proof}

\begin{lem}\label{homisiso} Let $f:G_1 \to G_2$ be an isogeny between reductive $k$-groups. Suppose $f$ restricts to an isomorphism between each corresponding pair of simple components and between maximal tori. Then $f$ is an isomorphism.
\end{lem}
\begin{proof} Since $G_1$ is reductive, its center $Z(G_1)$ is of multiplicative type. Hence $Z(G_1) \cap\ker f$ is trivial, as it is contained in every maximal torus of $G_1$. Assume $\ker f$ is non-trivial. Then by assumption $\ker f$ must be a non-central diagonally embedded subgroup of at least two of the simple factors of $G_1$. However these simple factors are all pairwise commuting, so $\ker f$ centralises all of them, so it must be central. We have a contradiction. 
\end{proof}

Henceforth let $G$ be a pseudo-reductive $k$-group, let $k'/k$ be the minimal field of definition for $\mathscr{R}_u(\smash{G_{\ovl{k}}})$, and let $\pi':G_{k'} \to G_{k'}/\mathscr{R}_u(G_{k'})=:G'$ be the natural projection. Recall that a smooth subgroup $L$ of $G$ is an almost Levi subgroup if $G(\ovl{k})=\mathscr{R}_u(G(\ovl{k})) \rtimes L(\ovl{k})$  (Definition \ref{defn:almostLevi}). The following result shows that this notion is only slightly more general than that of a Levi subgroup.

\begin{prop}\label{almostLevichar} 
Let $L$ be an almost Levi subgroup of $G$. 
Then $L$ is reductive, and $\pi'$ restricts to an isomorphism $Z(L_{k'})^{\circ} \to Z(G')^{\circ}$. 
Let $S$ be a pseudo-simple component of $G_{k^{\sep}}$. 
Then either 
\begin{itemize}
\item[(i)] $L_{k^{\sep}} \cap S$ is a Levi subgroup of $S$, or 
\item[(ii)] $\Char(k)=2$ and there exists $n:=n(S) \geq 1$ such that $S$ is of type $BC_n$, $L_{k^{\sep}} \cap S \cong \mathrm{SO}_{2n+1}$ and 
$\smash{L_{\ovl{k}} \cap S_{\ovl{k}} \cap \ker \pi'_{\ovl{k}} \cong (\alpha_2)^{2n}}$.
\end{itemize} 
In particular, if the root system of $G_{k^{\sep}}$ is reduced -- for instance if $\Char(k)\neq 2$ -- then $L$ is a Levi subgroup of $G$.
\end{prop}

\begin{proof} Without loss of generality we can assume that $k=k^{\sep}$, 
since being a Levi subgroup or an almost Levi subgroup is invariant under base change by $k^{\sep}/k$. 
Thus we may assume that $G$ is pseudo-split.

 We first show that $L$ is reductive. By definition of an almost Levi subgroup and since $G'$ is smooth, the restriction map $\smash{\pi'|_{L_{k'}}:L_{k'} \to G'}$ is bijective on $\smash{\ovl{k}}$-points and its kernel is infinitesimal and unipotent. In other words, $\pi'|_{L_{k'}}$ is a unipotent isogeny. Then $L$ is connected, since $\smash{\ker \pi'|_{L_{k'}}}$ is connected and as connectedness is preserved by group extensions. Since $\smash{\mathscr{R}_u(G_{\ovl{k}}) \cap L_{\ovl{k}}}$ has trivial $\smash{\ovl{k}}$-points and as $\smash{\pi'_{\ovl{k}}(\mathscr{R}_u(L_{\ovl{k}})) \subseteq \mathscr{R}_u(G'_{\ovl{k}})=1}$, it follows that $L$ is reductive.

 Next consider the central torus $Z(L_{k'})^{\circ}$ of $L_{k'}$. Since $\smash{\pi'|_{L_{k'}}}$ is an isogeny between reductive $k'$-groups, it restricts to an isogeny $Z(L_{k'})^{\circ} \to \pi'(Z(L_{k'})^{\circ})=Z(G')^{\circ}$. But $\ker\pi'$ is unipotent,
so $Z(L_{k'})^{\circ} \to Z(G')^{\circ}$ is an isomorphism.

 Let $S$ be a pseudo-simple component of $G$. Since $\smash{\pi'|_{L_{k'}}}$ is an isogeny between split reductive $k'$-groups, it induces a bijection between the respective sets of simple components of $L_{k'}$ and $G'$. 
Consider the simple component $S':=\pi'(S_{k'})$ of $G'$, and let $S_0$ be the simple component of $L$ satisfying $\pi'((S_0)_{k'})=S'$. Certainly $S_0 \subseteq S$. Let us define $\smash{\upsilon:(S_0)_{k'} \to S'}$ to be the composition of the inclusion $(S_0)_{k'} \hookrightarrow S_{k'}$ with $\pi'|_{S_{k'}}$. If $\upsilon$ is an isomorphism then $S_0$ is a Levi subgroup of $S$.

 Now assume that $\upsilon$ is not an isomorphism; then it is a unipotent isogeny between split simple $k$-groups.
 Such isogenies were classified in \cite[Lem.\ 2.2]{PY} and \cite[Thm.\ 2.2]{Va}.
 These results tell us that $\Char(k)=2$ and there exists $n:=n(S) \geq 1$ such that $S_0 \cong \mathrm{SO}_{2n+1}$, $S' \cong \mathrm{Sp}_{2n}$ and $\ker \upsilon = (S_0)_{k'} \cap \ker \pi' \cong (\alpha_2)^{2n}$. Let $U$ be the (direct) product of the first Frobenius kernels of the short root subgroups of $S_0$. Then $\ker \upsilon =U_{k'}$ by \cite[Lem.\ 2.2]{BRSS}. Consider the map $i_S:S\to R_{k'/k}(S')$ defined in $(\ref{iGmap})$ (certainly $\smash{\mathscr{R}_u(S_{\ovl{k}})}$ descends to $k'$, even if it is not the minimal such field). Observe that $U \subseteq \ker i_S$, but $U$ is non-central in $S$, and hence by \cite[6.2.15]{CP1} the root system of $S$ is non-reduced. The only irreducible non-reduced root system of rank $n$ is $BC_n$.

It remains to show that $S_0=L \cap S$. Let $T$ be a maximal torus of $L$. Then $T \cap S_0$ is a maximal torus of $S_0$ by \cite[A.2.7]{CGP}. Observe that $L$ has maximal rank in $G$. Hence, by a similar argument, $T \cap S=T \cap S_0$ is a maximal torus of $S$. Since $L$ is reductive, we have $Z_{L \cap ST}(T)=ST \cap Z_L(T)=T$. Then $$\Lie(L \cap ST) = \Lie(T) \oplus \bigoplus\{\mathfrak{g}_{\alpha}\hspace{0.5mm}|\hspace{0.5mm}\alpha \in \Phi(L,T) \cap \Phi(ST,T)\} = \Lie(S_0T).$$ That is, the inclusion $S_0T \subseteq L \cap ST$ induces an equality on Lie algebras. But $S_0T$ is smooth and connected, and so $S_0T = (L \cap ST)^{\circ}$. Taking the derived subgroup gives us $S_0=(L \cap S)^{\circ}$. Since $L_{k'} \cap S_{k'} \cap \ker \pi'$ has trivial $\smash{\overline{k}}$-points, it is infinitesimal and in particular connected. But $L_{k'} \cap S_{k'}$ is an extension of the connected $k'$-group $S'$ by $L_{k'} \cap S_{k'} \cap \ker \pi'$, and so $L \cap S$ is also connected. This shows that the third property in (ii) holds.

It remains to prove the second assertion of the proposition. Suppose the root system of $G_{k^{\sep}}$ is reduced. By the (already proved) first assertion, the isogeny $\smash{\pi'_{\ovl{k}}|_{L_{\ovl{k}}}}:L_{\ovl{k}} \to \smash{G'_{\ovl{k}}}$ restricts to an isomorphism between each corresponding pair of simple components. By \cite[11.14]{Bo} and since $\ker \pi'$ is unipotent, $\pi'_{\ovl{k}}|_{L_{\ovl{k}}}$ sends any maximal torus of $L_{\ovl{k}}$ isomorphically onto a maximal torus of $\smash{G'_{\ovl{k}}}$. Hence $\pi'_{\ovl{k}}|_{L_{\ovl{k}}}$ is an isomorphism by Lemma \ref{homisiso}. That is, $L$ is a Levi subgroup of $G$.
\end{proof}

We can now prove Theorem \ref{thm:levitating}.

\begin{proof}[Proof of Theorem \ref{thm:levitating}] 
Recall that $G$ is a pseudo-reductive $k$-group, and $H$ is a smooth subgroup of $G$.

 We first prove (i). We assume that $k=k^{\sep}$. In particular, $G$ and $H$ are pseudo-split. 

 If $H$ contains an almost Levi subgroup $L$ of $G$ then the restriction of $\pi'(\ovl{k})$ to $L(\ovl{k})$ is an isomorphism, so indeed $\pi'(H_{k'})= G'$. For the converse, suppose that $H$ is a levitation of $G'$ in $G$. 
 We will show that $H$ is pseudo-reductive so that we can then apply \cite[3.4.6]{CGP}.

Let $U$ be a smooth connected unipotent normal subgroup of $H$. Then $\pi'(U_{k'})$ is trivial as $G'$ is reductive, so $U_{k'} \subseteq \ker\pi'=\mathscr{R}_u(G_{k'})$. Consider the subgroup of $G_{k^{\sep}}$ generated by all $G(k^{\sep})$-conjugates of $U_{k^{\sep}}$; it admits a $k$-descent which we call $N$. Clearly $N$ is smooth, connected and unipotent, and $N_{k'}$ is contained in $\mathscr{R}_u(G_{k'})$. Moreover, $N$ is normal in $G$ since $G(k^{\sep})$ is dense in $G$. Hence $N$ is contained in $\mathscr{R}_u(G)$, which is trivial. So indeed $H$ is pseudo-reductive.
 
Since $H$ is pseudo-split and pseudo-reductive, applying \cite[3.4.6]{CGP} tells us there exists a Levi subgroup $L$ of $H$. Consider the maximal reductive quotient map $\kappa':H_{\ovl{k}} \to H_{\overline{k}}/\mathscr{R}_u(H_{\overline{k}})$. Note that $\kappa'$ sends $L_{\ovl{k}}$ isomorphically onto $H_{\overline{k}}/\mathscr{R}_u(H_{\overline{k}})$. Since $H$ is a levitation of $G'$ in $G$, the restriction $\smash{\pi'_{\ovl{k}}|_{H_{\ovl{k}}}}:H_{\ovl{k}} \to \smash{G'_{\ovl{k}}}$ is a quotient map. Hence $\smash{\pi'_{\ovl{k}}|_{H_{\ovl{k}}}}$ factors through $\kappa'$, as $G'$ is reductive. It follows that $\pi'(\ovl{k})$ sends $L(\ovl{k})$ onto $G'(\ovl{k})$. Since $L$ is reductive, the normal subgroup $L(\ovl{k})\cap \mathscr{R}_u(G(\ovl{k}))$ of $L(\ovl{k})$ must be trivial by Lemma \ref{unipmap}. Hence $L$ is an almost Levi subgroup of $G$. This completes the proof of (i).

 We next prove (ii). Since smoothness is a geometric property, we can assume that $k=k^{\sep}$. Suppose the root system of $G$ is reduced. Assume that $H$ is a levitation of $G'$ in $G$. Then, by part (i) along with Proposition \ref{almostLevichar}, $H$ contains a Levi subgroup $L$ of $G$. The restriction $\pi'|_{L_{k'}}:L_{k'} \to G'$ is an isomorphism, and so it induces an isomorphism on Lie algebras. Since $H$ contains $L$, the restriction $\pi'|_{H_{k'}}:H_{k'} \to G'$ induces a surjection on Lie algebras. Hence $\pi'|_{H_{k'}}$ is smooth by \cite[1.63]{Mi}.
\end{proof}

 The following example shows that both parts of Theorem \ref{thm:levitating} can fail if we allow $G$ to be quasi-reductive rather than pseudo-reductive.

\begin{example}\label{quasiredfails} 
Let $k$ be an imperfect field of characteristic $2$, and let $N=\alpha_2 \times \alpha_2$. Choose any smooth connected wound unipotent $k$-group $U$ in which $N$ embeds as a central subgroup. Let $G$ be the central product $(\PGL_2 \times U)/N$, where $N$ embeds into $\PGL_2$ as the intersection of its Frobenius kernel with its root groups. The unipotent radical $\mathscr{R}_u(G)$ of $G$ is isomorphic to $U$ --- in particular, $G$ is quasi-reductive --- and $\mathscr{R}_u(G)$ is a $k$-descent of $\mathscr{R}_u(G_{\overline{k}})$. So our map $\pi'$ is simply the natural projection $G \to G/\mathscr{R}_u(G)=:G'$. Let $H$ be the canonical subgroup of $G$ that is isomorphic to $\PGL_2$, and let $H':=\pi'(H)$. Then $H' = G' \cong \SL_2$, yet $G$ does not admit a Levi subgroup. Moreover, the restriction of $\pi'$ to $H$ has non-smooth kernel.
\end{example}

\section{Levitating subgroups}
\label{sec:levitatingsection}

 Let $G$ be a pseudo-reductive $k$-group with minimal field of definition $k'/k$ for its geometric unipotent radical, and quotient map $\pi':G_{k'} \to G_{k'}/\mathscr{R}_u(G_{k'}):=G'$.
%
In this section we investigate the following general question: \emph{which subgroups of $G'$ levitate in $G$, and which smooth subgroups of $G'$ levitate smoothly in $G$?} 
Along the way we prove Theorem \ref{thm:levitating_crit}.
Before proceeding, we give a brief overview of the results below to help with navigation.  Our basic idea is to consider three progressively more general cases: first, where $G = R_{k'/k}(G')$ is a Weil restriction; second, where $i_G(G)$ contains $\smash{\mathscr{D}(R_{k'/k}(G'))}$; finally, we consider the general case, for arbitrary pseudo-reductive $G$.

In the first case, if $H'$ is a smooth subgroup of $G'$ then Lemma \ref{adjunctionprop} tells us that $H'$ levitates in $G$, its largest levitation being $R_{k'/k}(H')$, which moreover is smooth. However Example \ref{ex:nonsmooth_nonlev} shows that a non-smooth subgroup $H'$ of $G'$ need not levitate in $G$.

In the second case,  it turns out that the condition that $i_G(G)$ contains $\smash{\mathscr{D}(R_{k'/k}(G'))}$ is equivalent to requiring that all tori of $G'$ levitate in $G$. 
If, in addition, the root system of $G_{k^{\sep}}$ is reduced, this condition is equivalent to requiring that all tori of $G'$ \emph{smoothly} levitate in $G$. 
This is a consequence of Theorem \ref{thm:levitating_crit} and its proof. A similar statement holds if we require that $H'$ is generated by tori; see Corollary \ref{thm:levitating_critcor}. However, unlike when $G=R_{k'/k}(G')$, a smooth subgroup $H'$ of $G'$ need not levitate in $G$; 
refer to Example \ref{ex:smooth_nonlev}. 

Finally consider the general case, for arbitrary pseudo-reductive $G$. If $H'$ is a smooth connected normal subgroup of $G'$ then it smoothly levitates in $G$ by \cite[3.1.6]{CGP}. However in the absence of normality a torus of $G'$ need not levitate in $G$, and even if it does, it need not smoothly levitate in $G$. Refer to Examples \ref{ex:non_levitating} and \ref{BC1torusnotlift}, respectively.

\subsection{Basic results and examples}
Our first basic result shows that in case a subgroup does levitate, there is a largest levitation.

\begin{prop}\label{largestlevitation} Let $H'$ be a subgroup of $G'$. If there exists at least one levitation (resp. smooth levitation) of $H'$ in $G$, then there exists a largest such levitation (resp. smooth levitation) of $H'$ in $G$: namely $\smash{i_G^{-1}(R_{k'/k}(H'))}$ (resp. $\smash{(i_G^{-1}(R_{k'/k}(H')))^{\Sm}}$).
\begin{proof} Assume that $H$ is a levitation of $H'$ in $G$. Combining Lemmas \ref{adjunctionprop} and \ref{adjunctionpropnew} tells us that $i_G(H)$ is contained in $R_{k'/k}(H')$, and that $q_{G'}(R_{k'/k}(H')_{k'}) = H'$. That is, $R_{k'/k}(H')$ is the largest subgroup $Z$ of $R_{k'/k}(G')$ that satisfies $q_{G'}(Z_{k'})=H'$. Since $\pi'=q_{G'}\circ (i_G)_{k'}$ by (\ref{pi'_eqn}), it follows that $i_G(H) \subseteq R_{k'/k}(H')$. 
So $H$ is contained in $\smash{i_G^{-1}(R_{k'/k}(H'))=:\widetilde{H}}$, which satisfies $\pi'(\smash{\widetilde{H}}_{k'})=H'$. 
In other words $\smash{\widetilde{H}}$ is the largest levitation of $H'$ in $G$. If $H$ is smooth (which implies that $H'$ is smooth), then $H$ is contained in $\widetilde{H}$, so $H\subseteq (\smash{\widetilde{H}})^{\Sm}$.  Hence $(\smash{\widetilde{H}})^{\Sm}$ is the largest smooth levitation of $H'$ in $G$.
\end{proof}
\end{prop}

An obvious example of a smooth subgroup of $G'$ that smoothly levitates in $G$ is the trivial subgroup; its largest levitation in $G$ is $\ker i_G$, and its largest smooth levitation is \'etale.

Now consider the (abstract group) homomorphism given by composing the canonical inclusion $G(k)\hookrightarrow G(k')$ with $\pi'(k'):G(k') \to G'(k')$; we abuse notation and call this map 
\begin{equation}\label{abusemap}\pi'(k):G(k) \to G'(k').\end{equation} 
We give a basic criterion for subgroups of $G'$ to levitate in $G$.

\begin{prop}\label{conjlevitations} Let $H'$ be a subgroup of $G'$ that levitates (resp. smoothly levitates) in $G$. Let $H'_1$ be another subgroup of $G'$. If there exists $g'$ in the image of $\pi'(k)$ such that $H'_1=g'H'(g')^{-1}$ then $H'_1$ levitates (resp. smoothly levitates) in $G$. 
\begin{proof} Let $g \in G(k)$ such that $\pi'(k)(g)=:g'$ satisfies $H'_1=g'H'(g')^{-1}$. Let $H$ be a levitation of $H'$ in $G$. Then $$\pi'((gHg^{-1})_{k'})=g'H'(g')^{-1}=H'_1,$$ so $gHg^{-1}$ is a levitation of $H'_1$ in $G$. If $H$ is smooth then of course so is $gHg^{-1}$.
\end{proof}
\end{prop}

It is remarkable that a partial converse to Proposition \ref{conjlevitations} holds, when $H'$ is a maximal rank smooth subgroup of $G'$. 
We prove this in Theorem \ref{conjlevitationsconverse}, and apply it to maximal tori in Corollary \ref{conjlevitationsconversecor}. In Theorem \ref{thm:levitating_regsubgroups} we find a sufficient condition for a smooth subgroup $H'$ of $G'$ to smoothly levitate in $G$, namely, there exists a maximal torus $T'$ of $G'$ which normalises $H'$ and smoothly levitates in $G$. If in addition the root system of $G_{k^{\sep}}$ is reduced we can relax this condition slightly, requiring only that $T'$ levitates in $G$ rather than smoothly levitates. We prove this in Corollary \ref{thm:levitating_regsubgroupscor}.

We continue with examples which exhibit various possible behaviours.

\begin{rem}[\textbf{$k$-structures on $G'$}]\label{simplerem} Suppose $k=k^{\sep}$. Fix a maximal torus $T$ of $G$. Choose a Levi subgroup $L$ of $G$ containing $T$; this exists by \cite[3.4.6]{CGP} as $G$ is pseudo-split. 
Then $G'$ inherits a canonical $k$-structure from $L$, via base changing by $k'/k$ and applying the isomorphism $\pi'|_{L_{k'}}:L_{k'} \to G'$. 
Denote $T':=\pi'(T_{k'})$. Observe that $T'$ is a $k$-defined maximal torus of $G'$ with respect to this $k$-structure. 
A $k$-structure on a split reductive $k'$-group is completely determined by a choice of maximal torus and a pinning (see for instance \cite[A.4.13]{CGP}). 
So, up to a choice of simple roots and a scaling of the root groups, this $k$-structure on $G'$ does not depend on the choice of $L$ for fixed $T$.

By \cite[23.39]{Mi}, any two pinnings of the split reductive pair $(G',T')$ are conjugate by some element of $N_{G'}(T')(k')$. Now let $T_1$ be another maximal torus of $G$. Since $k=k^{\sep}$, there exists $g\in G(k)$ such that $gTg^{-1}= T_1$. Denote $g':=\pi'(k)(g) \in G'(k')$. Hence any $k$-structure on $G'$ induced by $T_1$ is conjugate by $g'$ to one that is induced by $T$. However, if $\pi'(k)$ does not give rise to a surjection onto $G'(k')/T'(k')$ then not every $k$-structure on $G'$ is induced by a choice of maximal torus of $G$. So there is a set of privileged $k$-structures on $G'$, each of which arises from a choice of maximal torus of $G$.
\end{rem}

In
the following (very general) example we present a family of subgroups $H'$ of $G'$ that levitate in $G$; moreover, they smoothly levitate if $H'$ is smooth.

\begin{example}\label{simpleexam} Suppose $k=k^{\sep}$. Let $H'$ be a subgroup of $G'$.  Fix a maximal torus $T$ of $G$. Assume $H'$ is $k$-defined with respect to the $k$-structure on $G'$ induced by $T$ (as in Remark \ref{simplerem}). Choose any Levi subgroup $L$ of $G$ containing $T$. Then, under the identification $G'=L_{k'}$, there exists an algebraic $k$-group $J$ and an inclusion $J \hookrightarrow L$ whose base change by $k'/k$ is precisely $H' \hookrightarrow G'$. Since the composition $$G'=L_{k'} \hookrightarrow G_{k'} \xrightarrow{\pi'} G'$$ is the identity, $J$ is indeed a levitation of $H'$ in $G$. If $H'$ is smooth then $J$ is also smooth, as smoothness is invariant under base change. 

\end{example}

 Not all subgroups $H'$ of $G'$ that levitate in $G$ can be constructed as in Example \ref{simpleexam}, even if $G$ is pseudo-split. We illustrate this as follows. 

\begin{example}\label{nondescend} Let $k$ be a local function field of characteristic $2$. Let $k'$ be a degree 2 extension of $k$, and let $q':V' \to k'$ be an anisotropic non-degenerate quadratic form over $k'$ in 3 variables. It is shown in \cite[Ex.\ 7.2.2]{CGP} that $H':=\mathrm{SO}(q')$ does not descend to $k$. Consider the natural inclusion of $H'$ in $\mathrm{SL}(V')=:G'$. 
Denote $\smash{H:=R_{k'/k}(H')}$ and $\smash{G:=R_{k'/k}(G')}$. 
Since $H'$ is smooth, Lemma \ref{adjunctionprop} tells us that $H$ is a levitation of $H'$ in $G$. Let $L$ be a Levi subgroup of $G$ (which exists as $G$ is pseudo-split). 
If we had $\pi'((H \cap L)_{k'})=H'$ then $H \cap L$ would be a Levi subgroup of $H$, which violates \cite[7.2.1]{CGP}.
\end{example}

\subsection{Levitating tori}\label{levitatingtorisubsection}  



 Let $T'$ be a torus of $G'$. It is natural to ask whether $T'$ levitates to a torus of $G$.  Obstructions to this come in two flavours. First, $T'$ need not levitate at all in $G$. Consider the following example.

\begin{example}\label{ex:non_levitating} Let $G_1$ be a reductive $k$-group which is not a torus. Let $k'/k$ be a non-trivial purely inseparable finite field extension. Set $G_1':=(G_1)_{k'}$ and $G:=G_1\times R_{k'/k}(\Gm)$. Note that $\ker i_G$ is trivial. The minimal field of definition for $R_u(G_{\ovl{k}})$ is $k'/k$, the maximal reductive $k'$-quotient of $G$ is $G_1'\times \Gm=:G'$, and $L:= G_1\times \Gm$ is the unique Levi subgroup of $G$.  As explained in Remark~\ref{simplerem}, $G'$ inherits a $k$-structure from $L$.
Choose a maximal torus $T'$ of $G'$ which is not $k$-defined.  Suppose $T'$ levitates.  Then $T'$ levitates to a torus $T$ of $G$ by Lemma~\ref{thm:levitating_tori} below, since the root system of $G$ is reduced.  But $T$ is contained in $L$, so $T'$ must be $k$-defined by Example~\ref{simpleexam}, a contradiction.
\end{example}

 Even if $G$ is absolutely pseudo-simple, it need not be the case that $T'$ levitates in $G$. It turns out that one can find such a non-levitating torus if and only if $G$ is non-standard; this will be proved in Theorem \ref{thm:levitating_crit}.

 The second obstruction is as follows: even if $T'$ levitates in $G$, it need not smoothly levitate in $G$. We demonstrate this in Example \ref{BC1torusnotlift} (via Lemma \ref{nonliftingtorus}), for $G$ pseudo-simple of type $BC_1$. 
 In fact, this obstruction can only occur if the root system of $G_{k^{\sep}}$ is not reduced: if the root system is reduced then we prove in Lemma \ref{thm:levitating_tori} that any torus of $G'$ that levitates in $G$ must smoothly levitate in $G$. 

\begin{lem}\label{nonliftingtorus} 
Assume that $k=k^{\sep}$. Let $T$ be a maximal torus of $G$. Suppose there exists some root $\alpha \in \Phi(G,T)$ such that the restriction $U_{\alpha} \to i_G(U_{\alpha})$ is not surjective on $k$-points. Then there exists a maximal torus (resp., Borel subgroup) of $G'$ that does not levitate to any maximal torus (resp., minimal pseudo-parabolic subgroup) of $G$. 
\end{lem}

\begin{proof} 
Choose $g \in (i_G(U_{\alpha}))(k)$ such that $g \notin i_G(U_{\alpha}(k))$. Set $T_0:=i_G(T)$, $S_0:=gT_0g^{-1}$ and $S':=q_{G'}((S_0)_{k'})$. Observe that $S_0$ is a smooth levitation of $S'$ in $i_G(G)$. 
Recall that $\pi'=q_{G'}\circ (i_G)_{k'}$ by (\ref{pi'_eqn}), so $i_G^{-1}(S_0)$ is a levitation of $S'$ in $G$. 

 Assume (for a contradiction) that $S'$ levitates to some maximal torus $S$ of $G$. Since $R_{k'/k}(S')$ is commutative it contains a unique maximal torus, which must equal $S_0$. Again using the universal property of $q_{G'}$, we deduce that $i_G(S)=S_0$.

 Since $i_G(S)=S_0=gT_0g^{-1} \subseteq T_0i_G(U_{\alpha})$ we have that $S \subseteq TU_{\alpha}\ker i_G$, and hence $S \subseteq TU_{\alpha}$ since $(\ker i_G)(k)$ is finite. 
Applying \cite[3.2.1]{CGP} to the map $TU_{\alpha} \to T_0i_G(U_{\alpha})$ got by restricting $i_G$ tells us that the induced map $$N_{TU_{\alpha}}(T)/Z_{TU_{\alpha}}(T) \to N_{T_0i_G(U_{\alpha})}(T_0)/Z_{T_0i_G(U_{\alpha})}(T_0)$$ is an isomorphism. Consequently there exists $x \in U_{\alpha}(k)$ such that $S=xTx^{-1}$ and such that $g^{-1}i_G(x)$ centralises $T_0$. But $i_G(U_{\alpha})\cap Z_{i_G(G)}(T_0)$ is trivial, and so $g=i_G(x)$. This contradicts our assumption that $g \notin i_G(U_{\alpha}(k))$. That is, $S'$ does not levitate to any maximal torus of $G$.

 Now let $P$ be a minimal pseudo-parabolic subgroup of $G$ that contains both $T$ and $U_{\alpha}$. Then $P_0:=i_G(P)$ is a minimal pseudo-parabolic subgroup of $i_G(G)$ that contains the maximal torus $S_0$. Since $i_G(G)$ is pseudo-reductive, we can choose a cocharacter $\lambda:\mathbb{G}_m \to i_G(G)$ such that $P_0=P_{i_G(G)}(\lambda)$ and $Z_{i_G(G)}(S_0)=Z_{i_G(G)}(\lambda)$. Consider the minimal pseudo-parabolic subgroup $Q_0:=P_{i_G(G)}(-\lambda)$ of $i_G(G)$; by construction $P_0 \cap Q_0 = Z_{i_G(G)}(S_0)$. 
Define another cocharacter $\lambda':=q_{G'} \circ \lambda_{k'}:\mathbb{G}_m \to G'$, and consider the associated Borel subgroups $P':=P_{G'}(\lambda')$ and $Q':=P_{G'}(-\lambda')$ of $G'$. 
By construction $P' \cap Q' = S'$. Observe that $Q'=q_{G'}((Q_0)_{k'})$ by \cite[2.1.4, 2.1.9]{CGP}. In other words $Q'$ smoothly levitates to $Q_0$ in $i_G(G)$. So $Q'$ levitates to $i_G^{-1}(Q_0)$ in $G$.

 Assume (for a contradiction) that $Q'$ levitates to some minimal pseudo-parabolic subgroup $Q$ of $G$. Observe that $P\cap Q$ has maximal rank in $G$ by \cite[3.5.12(1)]{CGP}, hence $\pi'((P\cap Q)_{k'})=S'$. This implies that any maximal torus of $P \cap Q$ is a levitation of $S'$, which is a contradiction.
\end{proof}

 It follows that neither the maximal torus $S'$ nor the Borel subgroup $Q'$ of $G'$ constructed in the proof of Lemma \ref{nonliftingtorus} levitates to any smooth subgroup of $G$. One may either show this directly, or appeal to Theorem \ref{thm:levitating_regsubgroups}(ii).

 We now use Lemma \ref{nonliftingtorus} to give a concrete example of a type $BC_1$ $k$-group $G$ and a torus and Borel subgroup of $G'$ both of which levitate, but do not smoothly levitate, in $G$.

\begin{example}\label{BC1torusnotlift} Let $k$ be a separably closed imperfect field of characteristic $2$. Let $K= k(a^{1/2})$, where $a\in k$ but $a^{1/2}\not\in k$. Then $K= \{a^{1/2}x+ y\,|\,x,y\in k\}$.  Let $V'$ be the $k$-subspace $ka^{1/2}$ of $K$. Now take $V= K$ and define $q\colon V\to K$ by $q(z)= z^2$.  Set $\smash{V^{(2)}= q(V)= K^2\subseteq k}$. Then $V'\cap V^{(2)}= 0$.  Note also that $V^{(2)}= K^2$ has dimension $1$ as a vector space over $K^2$, and the subfield of $K$ generated over $k$ by $V'\oplus V^{(2)}$ is $K$ itself.
 
 Using this data, we can form a pseudo-simple $BC_1$ group $G$ of minimal type as in \cite[\S 9.8]{CGP} or \cite[Def.\ 3.3]{BRSS}. 
Let $\alpha$ be a very short root (multipliable) for $G$ with respect to a maximal torus $T$ of $G$. We may identify $U_\alpha(G)(k)$ with $V\times V'$ (regarded as a $k$-vector space) and $U_{2\alpha}(i_G(G))(k)$ with $K$, and the map $f$ from $U_\alpha(G)$ to $U_\alpha(i_G(G))$ induced by $i_G$ is given on $k$-points by $(v,v')\mapsto q(v)+v'$ \cite[9.6.8, 9.6.9]{CGP}. 
 
 Now assume $[k:k^2]> 2$. We claim that $f$ is not surjective on $k$-points,
which follows as long as we can show that $V^{(2)}\oplus V'\neq K$;
in fact, we show that $k$ is not contained in the image of $f$.
Given $c\in k$, if we want $c = f(v,v') = q(v)+v'$, then we must have $v'=0$. 
Since $[k:k^2]>2$, there exist $c\in k$ such that $c\neq ax^2+y^2$ for any $x,y\in k$, and hence $c$ does not lie in the image of $f$. This proves the claim.

 Given the claim, by Lemma \ref{nonliftingtorus} and the subsequent remark, there exists a maximal torus and a Borel subgroup of $G' \cong \SL_2$ both of which levitate, but do not smoothly levitate, in $G$. 
\end{example}

 We now move on to the proof of Theorem \ref{thm:levitating_crit}. We will need the following lemmas.

\begin{lem}\label{lem:torus_surj} 
Let $Z$ be a commutative affine algebraic $k$-group, and let $Z_t$ denote its unique maximal torus.
If $f$ is a surjective homomorphism from $Z$ to a $k$-torus $T$ then $f(Z_t)= T$.
\end{lem}

\begin{proof} 
By \cite[Thm.\ 16.13(a)]{Mi} there is a subgroup $Z_s$ of $Z$ such that $Z_s$ is of multiplicative type and $Z/Z_s$ is unipotent. Now $f$ induces a surjective homomorphism from $Z/Z_s$ to $T/f(Z_s)$. Since $Z/Z_s$ is unipotent and $T/f(Z_s)$ is a torus, $T/f(Z_s)$ is trivial. Hence we may assume without loss of generality that $Z=Z_s$ is of multiplicative type. Then by \cite[Cor.\ 12.24]{Mi} there is a short exact sequence
\begin{equation} \label{eqn:mult_ses} 1\to Z_t\to Z\to F\to 1, \end{equation}
where $F$ is finite. Now $f$ induces a surjective homomorphism from $Z/Z_t$ to $T/f(Z_t)$.  Since $Z/Z_t$ is finite and $T/f(Z_t)$ is a torus, $T/f(Z_t)$ is trivial. Hence $f(Z_t)=T$.
\end{proof}

\begin{lem}\label{thm:levitating_tori} 
Let $G$ be a pseudo-reductive $k$-group. Suppose the root system of $G_{k^{\sep}}$ is reduced. Let $T'$ be a torus of $G'$ and assume that there exists some subgroup $Z$ of $i_G(G)$ such that $q_{G'}(Z_{k'})=T'$. Then there exists a unique torus $T$ of $G$ such that $\pi'(T_{k'})=T'$. If $T'$ is maximal in $G'$ then $T$ is maximal in $G$. 
\end{lem}

\begin{proof} 
Let $T'$ be a torus of $G'$. By assumption there exists a subgroup $Z$ of $i_G(G)$ such that $q_{G'}(Z_{k'})=T'$. Then $Z$ is contained in $R_{k'/k}(T')$ by Lemma \ref{adjunctionpropnew}; in particular $Z$ is commutative, so $Z$ contains a unique maximal torus $T_0$. By \cite[C.4.4]{CGP} this torus $T_0$ remains maximal in $Z$ after base change by $k'/k$. Then applying Lemma \ref{lem:torus_surj} says that $q_{G'}$ sends $(T_0)_{k'}$ onto $T'$. That is, $T_0$ is a smooth levitation of $T'$ in $i_G(G)$.

 By assumption the root system of $G_{k^{\sep}}$ is reduced, so $\ker i_G$ is central in $G$. Then, by the proof of \cite[2.2.12(1)]{CGP}, $i_G^{-1}(T_0)$ is commutative and its unique maximal torus $T$ satisfies $i_G(T)=T_0$. Since $\pi'=q_{G'}\circ (i_G)_{k'}$ we deduce that $\pi'(T_{k'})=T'$. Moreover, $T$ is the unique torus that satisfies this property, as any other such torus must also map onto $T_0$ via $i_G$. 

 Finally, observe that $\rank_k(G) = \rank_{k'}(G')$ as $k'/k$ is purely inseparable. The final assertion follows immediately.
\end{proof}

We can now prove Theorem \ref{thm:levitating_crit}.

\begin{proof}[Proof of Theorem \ref{thm:levitating_crit}]
 Suppose that the root system of $G_{k^{\sep}}$ is reduced.  We start by observing that each of (a)--(f) is invariant under replacing $k$ (resp., $k'$) with $k^{\sep}$ (resp., $(k')^{\sep}$).  For (a) and (b), note that the formation of the $i_G$ map commutes with separable algebraic field extensions (since this is true of the unipotent radical and minimal fields of definition, by \cite[1.1.9]{CGP}). 
Moreover, the formation of the derived group commutes with arbitrary field extensions by \cite[Cor.\ 6.19(a)]{Mi}. For (c)--(e), if $H'$ is a subgroup of $G'$ and $(H')_{k^{\sep}}$ levitates to a subgroup $M_1$ of $G_{k^{\sep}}$ then the subgroup $M$ of $G_{k^{\sep}}$ generated by the Galois conjugates of $M_1$ is a levitation of $(H')_{k^{\sep}}$ and $M$ descends to a subgroup $H$ of $G$ which is a levitation of $H'$; it then follows from Lemma~\ref{thm:levitating_tori} that if $H'$ is a (maximal) torus of $G'$ then $H'$ levitates to a (maximal) torus of $G$. For (f), it is clear from the standard construction (see \cite[1.4]{CGP}) that if $G$ is standard then $G_{k^{\sep}}$ is standard, and the converse is \cite[5.2.3]{CGP}; moreover, if $k’$ is the minimal field of definition for $\mathscr{R}_u(G_{\ovl{k}})$ then $(k’)^{\sep}$ is the minimal field of definition for $R_u((G_{k^{\sep}})_{\ovl{k}})$ by \cite[1.1.8]{CGP}, and one can show using a Galois descent argument that if $G$ is pseudo-simple and $S$ is a pseudo-simple factor of $G_{k^{\sep}}$ then the minimal fields of definition for $\mathscr{R}_u((G_{k^{\sep}})_{\ovl{k}})$ and $\mathscr{R}_u(S_{\ovl{k}})$ are equal.
 
Hence we can assume without loss that $k= k^{\sep}$; in particular, $G$ is pseudo-split.

 (a)$\iff$(b). 
Assume $i_G(G)$ contains $\mathscr{D}(R_{k'/k}(G'))$.  Then $q_{G'}$ is smooth and surjective since $G'$ is smooth, and $\ker q_{G'} = \mathscr{R}_u(R_{k'/k}(G')_{k'})$ since $G'$ is reductive. Let $T_0$ be a maximal torus of $R_{k'/k}(G')$. Denote $T':=q_{G'}((T_0)_{k'})$.  Let $Z$ be the unique maximal torus of $Z(R_{k'/k}(G'))$.  We claim that $Z\subseteq i_G(G)$.  To see this, observe that $i_G(G)$ is a pseudo-split pseudo-reductive $k$-group, hence it contains a Levi subgroup $L$ by \cite[3.4.6]{CGP}. Now $L$ is also a Levi subgroup of $R_{k'/k}(G')$ by \cite[9.2.1(2)]{CGP}, 
so it contains some maximal torus of $R_{k'/k}(G')$. But all such maximal tori are $R_{k'/k}(G')(k)$-conjugate since $k= k^{\sep}$, so $L$ contains the central torus $Z$, and the claim is proved.

We have $Z(R_{k'/k}(G'))= R_{k'/k}(Z(G')$ by \cite[A.5.15(1)]{CGP}, so $q_{G'}(Z(R_{k'/k}(G')))= Z(G')$ by Lemma~\ref{adjunctionpropnew}.  Hence $q_{G'}(Z_{k'})= Z(G')^0$ by Lemma~\ref{lem:torus_surj}.  Choose a subtorus $M$ of $T_0$ such that $MZ= T_0$ and $M\cap Z$ is finite.  Then $q_{G'}(M_{k'})Z(G')^0= q_{G'}(M_{k'})q_{G'}(Z_{k'})= q_{G'}((T_0)_{k'})= T'$ and $q_{G'}(M_{k'})\cap Z(G')^0$ is finite.  It follows that $q_{G'}(M_{k'})\subseteq \mathscr{D}(G')$.  To prove (b), it is enough by (a) to show that $M\subseteq \mathscr{D}(R_{k'/k}(G'))$.

Since $\ker q_{G'}$ is unipotent, combining $q_{G'}$ with base change by $k'/k$ induces an inclusion of root systems $\iota:\Phi(G',T') \hookrightarrow \Phi(R_{k'/k}(G'),T_0)$.  Let $\alpha' \in \Phi(G',T')$ and consider the associated $T'$-root group $U_{\alpha'}$ of $G'$. Let $\alpha:=\iota(\alpha')$ and consider the associated $T_0$-root group $U_{\alpha}$ of $R_{k'/k}(G')$. Note that $U_{\alpha'}$ is 1-dimensional as $G'$ is reductive, and hence $q_{G'}((U_{\alpha})_{k'})=U_{\alpha'}$. 
 Recall that the derived subgroup of a pseudo-reductive $k$-group is perfect \cite[3.1]{CGP}, and that it is generated by all of the root groups \cite[3.1.5]{CGP}.  Consequently $q_{G'}(\mathscr{D}(R_{k'/k}(G'))_{k'})=\mathscr{D}(G')$, 
and so \begin{equation}\label{littlebit} q_{G'}^{-1}(\mathscr{D}(G'))=\mathscr{D}(R_{k'/k}(G'))_{k'}\ker q_{G'}.\end{equation}
It follows that $M_{k'}\subseteq \mathscr{D}(R_{k'/k}(G'))_{k'}\ker q_{G'}$.  Since $\ker q_{G'}$ is smooth and unipotent, there is a maximal torus $T_1'$ of $\mathscr{D}(R_{k'/k}(G'))_{k'}\ker q_{G'}$ such that $T_1'\subseteq \mathscr{D}(R_{k'/k}(G'))_{k'}$.  But $k= k^{\sep}$ and $\mathscr{D}(R_{k'/k}(G'))_{k'}$ is normal in $R_{k'/k}(G')_{k'}$, so every maximal torus of $\mathscr{D}(R_{k'/k}(G'))_{k'}\ker q_{G'}$ is contained in $\mathscr{D}(R_{k'/k}(G'))_{k'}$.  Hence $M_{k'}\subseteq \mathscr{D}(R_{k'/k}(G'))_{k'}$.  It follows that $M\subseteq \mathscr{D}(R_{k'/k}(G'))$, so (b) holds.

On the other hand, since the derived subgroup $\mathscr{D}(R_{k'/k}(G'))$ of $R_{k'/k}(G')$ is perfect, it is generated by tori. So if (b) holds then so does (a). 

 (b)$\implies$(c). Let $S'$ be a torus of $G'$, and let $S_0$ denote the unique maximal torus of the smooth commutative group $R_{k'/k}(S')$. 
 Then $(S_0)_{k'}$ is the unique maximal torus of $R_{k'/k}(S')_{k'}$. Combining this with Lemmas \ref{adjunctionprop} and \ref{lem:torus_surj} tells us that $q_{G'}((S_0)_{k'})=S'$. 
By assumption (b) holds, so $S_0 \subseteq i_G(G)$. Since $\smash{\pi'=q_{G'}\circ (i_G)_{k'}}$, we deduce that $\smash{\pi'(i_G^{-1}(S_0)_{k'})=S'}$. In other words, $i_G^{-1}(S_0)$ is a levitation of $S'$ in $G$.

 (c)$\implies$(d). This follows from Lemma \ref{thm:levitating_tori}, since $\smash{\pi'=q_{G'}\circ (i_G)_{k'}}$.

 (d)$\implies$(e) is clear, as $\rank_k G=\rank_{k'} G'$. 

 (e)$\implies$(b). Assume (e) holds. Let $T_0$ be a maximal torus of $R_{k'/k}(G')$. Consider the maximal torus $q_{G'}((T_0)_{k'})=:T'$ of $G'$. By assumption there exists a maximal torus $T$ of $G$ such that $\pi'(T_{k'})=T'$. Observe that $q_{G'}(i_G(T)_{k'})=\pi'(T_{k'})=T'$, and hence $i_G(T) \subseteq R_{k'/k}(T')$ by Lemma \ref{adjunctionpropnew}. Similarly, $T_0 \subseteq R_{k'/k}(T')$. But $\rank_{k'} G'=\rank_k R_{k'/k}(G')=\rank_k G$, so $T_0=i_G(T)$ must be the unique maximal torus of $R_{k'/k}(T')$. As $T_0$ was chosen arbitrarily, we have shown that (b) holds.

 (a)$\iff$(f). 
If $G$ is an absolutely pseudo-simple $k$-group with reduced root system then $G$ is standard if and only if $i_G(G)= \mathscr{D}(R_{k'/k}(G'))$, by \cite[5.3.8]{CGP}.  Now consider the general case, for arbitrary pseudo-reductive $G$. We have a decomposition $G=\mathscr{D}(G)\cdot C$, where $C$ is any Cartan subgroup of $G$, and $\mathscr{D}(G)$ is a commuting product of (absolutely) pseudo-simple $k$-groups $S_i$ for $i=1,...,r$. Observe that $G$ is standard if and only if $S_i$ is standard for each $i$ (this follows from \cite[5.2.3, 5.2.6, 5.3.1]{CGP}).

 Henceforth fix some $i \in \{1,...,r\}$. The aforementioned decomposition of $G$ is preserved by the $i_G$ map. So we have a (co)restriction map $i_G|_{S_i}:S_i \to R_{k'/k}(S_i')$, where $S_i':=\pi'((S_i)_{k'})$. By functoriality $i_G|_{S_i}$ is the map associated to $\pi'|_{(S_i)_{k'}}:(S_i)_{k'} \to S_i'$ under adjunction.

 Assume (a) holds. Since the formation of the derived subgroup commutes with (commuting) products, 
and as $S_i$ is perfect, we see that $i_G|_{S_i}$ maps onto $\mathscr{D}(R_{k'/k}(S_i'))$. It follows that $k'/k$ is the minimal field of definition for $\mathscr{R}_u((S_i)_{\overline{k}})$, as otherwise surjectivity would fail by dimension considerations. 
So $i_G|_{S_i}=i_{S_i}$, i.e., it is the analogue of the $i_G$ map for $S_i$. Since $k=k^{\sep}$ we can apply \cite[5.3.8]{CGP}, so $S_i$ is standard. Hence $\mathscr{D}(G)$ is a commuting product of standard pseudo-simple $k$-groups, and so $G$ is itself standard.

 Conversely, assume (f) holds. Since $G$ is standard, $S_i$ is also standard. By assumption $k'/k$ is the minimal field of definition for $\mathscr{R}_u((S_i)_{\overline{k}})$, so we can again apply \cite[5.3.8]{CGP}, which tells us that $i_G|_{S_i}=i_{S_i}$ maps onto $\mathscr{D}(R_{k'/k}(S_i'))$. Again using the fact that the formation of the derived subgroup commutes with (commuting) products, it follows that $i_G(G)$ contains $\mathscr{D}(R_{k'/k}(G'))$.
\end{proof}


We can extend the results of Theorem~\ref{thm:levitating_crit} to subgroups of $G'$ that are generated by tori.

\begin{cor}\label{thm:levitating_critcor} Let $G$ be a pseudo-reductive $k$-group. Suppose the root system of $G_{k^{\sep}}$ is reduced, and that $G$ satisfies the equivalent conditions of Theorem \ref{thm:levitating_crit}. Let $H'$ be a subgroup of $G'$ that is generated by tori (for example, this holds if $H'$ is perfect). Then there exists a smooth subgroup $H$ of $G$ such that $\pi'(H_{k'})=H'$.
\end{cor}

\begin{proof}
This is an immediate consequence of Theorem \ref{thm:levitating_crit}(d); simply take a generating set of tori $\{T'_i \hspace{0.5mm}|\hspace{0.5mm} i =1,...,l\}$ of $H'$, levitate each $T'_i$ to a torus $T_i$ of $G$, and let $H$ be the subgroup of $G$ generated by $\{T_i \hspace{0.5mm}|\hspace{0.5mm} i =1,...,l\}$. The formation of $H$ commutes with base change by $k'/k$ by \cite[2.47]{Mi}, and $H$ is smooth by \cite[2.48]{Mi}, so indeed $H$ is a smooth levitation of $H'$ in $G$.
\end{proof}

\subsection{Levitating maximal rank subgroups}\label{levitatingmaxranksubgroupssubsection} 

Next we study levitations of maximal rank subgroups.  Recall the map $\pi'(k):G(k) \to G'(k')$ from (\ref{abusemap}).

\begin{thm}\label{conjlevitationsconverse} Suppose $k=k^{\sep}$. Let $G$ be a pseudo-reductive $k$-group. Let $H$ be a smooth subgroup of $G$ such that $\pi'(H_{k'})=:H'$ has maximal rank in $G'$. Let $H'_1$ be another subgroup of $G'$ that is $G'(k')$-conjugate to $H'$. If there exists a smooth subgroup $H_1$ of $G$ satisfying $\pi'((H_1)_{k'})=H'_1$ then there exists $g'$ in the image of $\pi'(k)$ such that $H'_1=g'H'(g')^{-1}$.
\end{thm}


\begin{proof}
Let $H$ be a smooth levitation of $H'$ in $G$. Combining \cite[C.4.4]{CGP} with \cite[11.14]{Bo} tells us that $H$ is a maximal rank subgroup of $G$. So let $T$ be a maximal torus of $H$; then $\pi'(T_{k'})=:T'$ is a maximal torus of $H'$. Similarly let $H_1$ be a smooth levitation of $H'_1$ in $G$, let $T_1$ be a maximal torus of $H_1$; then $\pi'((T_1)_{k'})=:T'_1$ is a maximal torus of $H'_1$.

 Since $k=k^{\sep}$ there exists $x \in G(k)$ such that $T_1=xTx^{-1}$. 
 Denote $x':=\pi'(k)(x)$. Consider the subgroups $H'$ and $(x')^{-1}H'_1x'$ of $G'$; by assumption they are $G'(k')$-conjugate, but they share a maximal torus $T'$, so there exists $n' \in N_{G'}(T')(k')$ such that $(x')^{-1}H'_1x'=n'H'(n')^{-1}$.

 We claim that $\pi'(k)$ induces a surjection $N_G(T)(k) \to N_{G'}(T')(k')/T'(k')$. Given the claim, there exists $n \in N_G(T)(k)$ such that $\pi'(k)(n)=n't'$ for some $t' \in T'(k')$. Let $g:=xn$ and let $g':=\pi'(k)(g)=x'n't'$; then indeed $H'_1=g'H'(g')^{-1}$. 

 It remains to prove the claim. We need the following ingredients. First observe that $N_G(T)/Z_G(T)$ is a constant $k$-group, so the canonical inclusion \begin{equation}\label{onee}(N_G(T)/Z_G(T))(k) \to (N_G(T)/Z_G(T))(k')\end{equation} is an isomorphism. Note that $Z_{G'}(T')=T'$, as $G'$ is reductive. Then by \cite[3.2.1]{CGP}, as $\ker\pi'$ is unipotent, $\pi'$ induces an isomorphism \begin{equation}\label{twoo}(N_G(T)/Z_G(T))(k') \to (N_{G'}(T')/T')(k').\end{equation} Next observe that $(N_G(T)/Z_G(T))(k)=N_G(T)(k)/Z_G(T)(k)$: this holds as $k$ is separably closed and $Z_G(T)=N_G(T)^{\circ}$ is smooth. Similarly $(N_{G'}(T')/T')(k')=N_{G'}(T')(k')/T'(k')$. Then composing the natural projection $N_G(T)(k)\to N_G(T)(k)/Z_G(T)(k)$ with $(\ref{onee})$ and $(\ref{twoo})$ gives us the desired surjection $N_G(T)(k) \to N_{G'}(T')(k')/T'(k')$. This completes the proof.
\end{proof}

\begin{cor}\label{conjlevitationsconversecor} Suppose $k=k^{\sep}$. Let $G$ be a pseudo-reductive $k$-group with reduced root system. Let $Z$ be a subgroup of $G$ such that $\pi'(Z_{k'})=:T'$ is a maximal torus of $G'$. Let $T'_1$ be another maximal torus of $G'$. Then there exists a subgroup $Z_1$ of $G$ satisfying $\pi'((Z_1)_{k'})=T'_1$ if and only if there exists $g'$ in the image of $\pi'(k)$ such that $T'_1=g'T'(g')^{-1}$.
\end{cor}

\begin{proof}[Proof of Corollary \ref{conjlevitationsconversecor}] 
Since $k'$ is separably closed, $T'$ and $T'_1$ are $G'(k')$-conjugate. By Lemma \ref{thm:levitating_tori}, $T'$ smoothly levitates in $G$. Similarly, $T'_1$ levitates in $G$ if and only if it smoothly levitates in $G$. The result then follows from combining Proposition \ref{conjlevitations} with Theorem \ref{conjlevitationsconverse}.
\end{proof}

\subsection{Levitating regular subgroups}\label{levitatinggeneralsubsection}

Recall that a subgroup of a pseudo-reductive group is \emph{regular} if it is normalised by some maximal torus. In this subsection we prove Theorem \ref{thm:levitating_regsubgroups}, which concerns smooth levitations of regular smooth subgroups. We first need the following lemmas.

\begin{lem}\label{secondlem} Let $G$ be a smooth connected affine $k$-group. Let $S$ be a torus of $G$. Then $Z_G(S) \cap \mathscr{R}_u(G)=\mathscr{R}_u(Z_G(S))$. 
\begin{proof} Let $\pi:G \to G/\mathscr{R}_u(G)=:\overline{G}$ be the natural projection, and denote $\overline{S}:=\pi(S)$. Observe that $\pi(Z_G(S))=Z_{\overline{G}}(\overline{S})$ by \cite[11.14, Cor.\ 2]{Bo}. Now $Z_{\overline{G}}(\overline{S})$ is pseudo-reductive by \cite[1.2.4]{CGP}, hence $\pi(\mathscr{R}_u(Z_G(S)))$ is trivial. So $\mathscr{R}_u(Z_G(S)) \subseteq \mathscr{R}_u(G)$. 

For the opposite inclusion, certainly $Z_G(S) \cap \mathscr{R}_u(G)$ is a unipotent normal subgroup of $Z_G(S)$. The same argument as in \cite[C.2.23, proof, par.\ 4]{CGP} tells us that $Z_G(S) \cap \mathscr{R}_u(G)=\mathscr{R}_u(G)^S$ is smooth and connected. Hence $Z_G(S) \cap \mathscr{R}_u(G) \subseteq \mathscr{R}_u(Z_G(S))$.
\end{proof}
\end{lem}

\noindent Henceforth we use the notation of the previous subsection. That is, $G$ is a pseudo-reductive $k$-group, $k'/k$ is the minimal field of definition for $\mathscr{R}_u(G_{\ovl{k}})$, and $\pi':G_{k'} \to G_{k'}/\mathscr{R}_u(G_{k'}):=G'$ is the natural projection.

\begin{lem}\label{descenttorus} Let $S'$ be a split torus of $G'$. Suppose that $S'$ levitates to a torus $S$ in $G$. Then any subgroup of $S'$ also levitates in $G$.
\begin{proof} Every split multiplicative type $k$-group of $M$ uniquely descends to $\Z$ (since $M \mapsto X(M)$ is an equivalence from the category of split multiplicative type $k$-groups of to the opposite category of finitely generated $\Z$-modules; see \cite[12.23]{Mi}). Similarly, any embedding of split multiplicative type $k$-groups uniquely descends to $\Z$.

 Let $M'$ be a subgroup of $S'$. By the above remark the embedding $S_{k'} \cap (\pi')^{-1}(M') \hookrightarrow S_{k'}$ uniquely descends to $\Z$ and hence to $k$. Let us call this $k$-descent $M \hookrightarrow S$. Then $\pi'$ sends $M_{k'}$ isomorphically onto $M'$, since $\ker \pi'$ is unipotent. 
\end{proof}
\end{lem}




\begin{proof}[Proof of Theorem \ref{thm:levitating_regsubgroups}] 
 Let $H'$ be a smooth subgroup of $G'$, and let $T'$ be a maximal torus of $G'$ that normalises $H'$. Assume that $T'$ levitates to a (maximal) torus $T$ of $G$.  We first prove (i). Using a standard argument of Galois descent, we may assume without loss of generality that $k=k^{\sep}$ (cf.\ the remarks about (c)--(e) in the proof of Theorem~\ref{thm:levitating_crit}); note that the formation of normalisers commutes with base change.

 Recall that maximal tori in reductive $k$-groups are self-centralising. Hence $$T' \subseteq \pi'(Z_G(T)_{k'}) \subseteq Z_{G'}(T')=T'.$$ That is, the Cartan subgroup $Z_G(T)$ is a levitation of $T'$ in $G$. Indeed $Z_G(T)$ is smooth by \cite[17.44]{Mi}.

 Since $T$ is split and $\pi'$ restricts to an isomorphism $T_{k'} \to T'$, extending scalars by $k'/k$ induces a natural bijection $\iota:X(T') \to X(T)$. One can check that $\iota$ restricts to an injection between the respective root systems $\Phi(G',T') \hookrightarrow \Phi(G,T)$, mapping onto the set of non-multipliable roots of $\Phi(G,T)$ (see for instance \cite[2.3.10]{CGP}). Since $H'$ is normalised by $T'$, $\iota$ further restricts to an injection $\Phi(H',T') \hookrightarrow \Phi(G,T)$. (Note that the set of roots $\Phi(H',T')$ is still defined even when $H'$ is not pseudo-reductive; it just may no longer be symmetric about the origin.)

 Since $H'$ is smooth and $T'$-stable, \cite[13.20]{Bo} tells us that $(H')^{\circ}$ is generated by the $T'$-root groups $U_{\alpha'}$ for each $\alpha' \in \Phi(H',T')$ along with the torus $(H' \cap T')^{\circ}_{\red}=:S'$. By Lemma \ref{descenttorus} (and its proof), there exists a subtorus $S$ of $T$ such that $S$ is a levitation of $S'$.

 Let $\alpha' \in \Phi(H',T')$ and consider the corresponding $T'$-root group $U_{\alpha'}$ of $H'$. Define a root $\alpha \in \Phi(G,T)$ as follows: if $\iota(\alpha')$ is non-divisible in $\Phi(G,T)$ then set $\alpha=\iota(\alpha')$, otherwise set $\alpha=\iota(\alpha')/2$. Let $U_{\alpha}$ be the $T$-root group of $G$ associated to $\alpha$. Observe that $\pi'((U_{\alpha})_{k'})$ is a $T'$-stable smooth connected subgroup of $G'$ whose Lie algebra is the $\alpha'$-root space in $\Lie (G')$. But this condition uniquely defines a $T'$-root group of $G'$ by \cite[2.3.11]{CGP}, and hence $\pi'((U_{\alpha})_{k'})=U_{\alpha'}$. That is, the root group $U_{\alpha}$ is a smooth levitation of $U_{\alpha'}$ in $G$.

 Let $\Psi$ be the subset of $\Phi(G,T)$ consisting of all roots $\alpha$ defined as above, that is $$\Psi:=\big\{\alpha \in \Phi(G,T)\hspace{0.5mm}\big|\hspace{0.5mm} \alpha \textnormal{ is non-divisible, }\iota^{-1}(\alpha) \in \Phi(H',T') \sqcup \tfrac{1}{2}\Phi(H',T')\big\}.$$ Let us define $H$ to be the subgroup of $G$ generated by the $T$-root groups $U_{\alpha}$ for each $\alpha \in \Psi$ along with the torus $S$. Note that $H$ is smooth and connected. Since the formation of $H$ commutes with base change \cite[2.47]{Mi}, we deduce that $\pi'(H_{k'})=(H')^{\circ}$. That is, $H$ is a levitation of $(H')^{\circ}$ in $G$. 

Now let $\Lambda$ be a set of representatives of the cosets of $(H')^{\circ}(k')$ in $H'(k')$; we can choose them so that $\Lambda$ normalises $T'$ (since all maximal tori in $H'T'$ are conjugate by an element of $(H')^{\circ}(k')$). For each $h \in \Lambda$, we define an element $g:=g(h) \in G(k)$ as follows. Since $\ker\pi'$ is unipotent, $\pi'$ restricts to a surjection $\pi'|_{N_G(T)_{k'}}:N_G(T)_{k'} \to N_{G'}(T')$ by \cite[3.2.1]{CGP}. Observe that $$(\ker \pi'|_{N_G(T)_{k'}})^{\circ} = (N_G(T)_{k'} \cap \ker\pi')^{\circ}= Z_G(T)_{k'} \cap \mathscr{R}_u(G_{k'})= \mathscr{R}_u(Z_G(T)_{k'}),$$ where the final equality is due to Lemma \ref{secondlem} (along with the fact that centralisers commute with base change). In particular, $\ker \pi'|_{N_G(T)_{k'}}$ is smooth. Hence, since $k'$ is separably closed, $\pi'$ induces a surjection $N_G(T)(k') \to N_{G'}(T')(k')$. So choose $g' \in N_G(T)(k')$ such that $\pi'(k')(g')=h$.

 We claim that $g'$ normalises $H$. To see this, we first observe that the formation of root groups of $G$ commutes with base change by $k'/k$. Since $g'$ normalises $T$ and $h$ normalises $H'$, 
it follows that $g'$ stabilises the set of $T$-root groups $\{U_{\alpha}\hspace{0.5mm}|\hspace{0.5mm}\alpha \in \Psi\}$. Note that $h$ normalises $S'$, 
hence $g'$ normalises $S$. 
So indeed $g' \in N_G(H)(k')$. 

 The Cartan subgroup $Z_G(T)$ normalises each $T$-root group $U_{\alpha}$ of $G$ for $\alpha \in \Psi$, and of course it centralises $S$, so $Z_G(T)$ is contained in $N_G(H)$.  Let $M= N_G(H)^{\Sm}$.  Clearly $T$ normalises $M$, so $M$ is generated by $H$, $Z_G(T)$ and the root groups $U_\alpha$ such that $\alpha$ is perpendicular to all the roots of $H$; hence $N_G(H)$ normalises $M$.  Applying \cite[17.47]{Mi} to $M$ tells us that $N_G(M)^0= M^0$, and we deduce that $N_G(H)$ is itself smooth. 
Moreover $N_G(H)/N_G(H)^{\circ}$ is constant as $k=k^{\sep}$. It follows that the inclusion $N_G(H)(k) \hookrightarrow N_G(H)(k')$ induces a surjection $N_G(H)(k) \to N_G(H)(k')/N_G(H)^{\circ}(k')$. 
So there exists $x' \in N_G(H)^{\circ}(k')$ such that $g'x' \in N_G(H)(k)$; denote $g=g(h):=g'x'$. Then the subgroup $\langle H, g(h) \hspace{0.5mm}|\hspace{0.5mm} h \in \Lambda \rangle$ of $G$ is a smooth levitation of $H'$ in $G$.

 We have shown that $H'$ smoothly levitates in $G$. Then Proposition \ref{largestlevitation} says that $H'$ admits a largest smooth levitation in $G$, which must be normalised by $T$ since all of its $T(k)$-conjugates are also smooth levitations of $H'$ in $G$. 
This completes the proof of (i).

 We next prove (ii). Observe that maximal tori, Cartan subgroups, root groups and pseudo-parabolic subgroups of $G$ are all examples of regular subgroups of $G$. Since the formation of all of the aforementioned subgroups of $G$ commutes with base change by separable field extensions, we can assume without loss of generality that $k=k^{\sep}$.

 In the course of proving (i), we showed that if $H$ is a Cartan subgroup (resp. root group) of $G$ then $\pi(H_{k'})$ is a maximal torus (resp. root group) of $G'$. 
If $H$ is a pseudo-parabolic subgroup of $G$ then $\pi(H_{k'})$ is a parabolic subgroup of $G'$ by \cite[3.5.4]{CGP}. This proves the ``if'' direction.

 It remains to prove the converse. Without loss of generality we can assume that $H'$ is connected. Henceforth take $H$ to be the largest smooth levitation of $H'$ in $G$; 
we know it exists by (i).

 We first assume that $H'$ is a maximal torus of $G'$, i.e., $H'=T'$. By Proposition \ref{largestlevitation}, $i_G^{-1}(R_{k'/k}(T'))$ is the largest levitation of $T'$ in $G$, so
 $$i_G(T) \subseteq i_G(H) \subseteq R_{k'/k}(T') \subseteq R_{k'/k}(G').$$
 Observe that $i_G(T)$ is a maximal torus of $R_{k'/k}(G')$, since $i_G(G)$ contains a Levi subgroup of $R_{k'/k}(G')$ by \cite[9.2.1(2)]{CGP}. Moreover, $R_{k'/k}(T')$ is a Cartan subgroup of $R_{k'/k}(G')$ by \cite[A.5.15(3)]{CGP}. Consequently $i_G(H)$ centralises $i_G(T)$. Now \cite[11.14, Cor.\ 2]{Bo} says that $i_G$ sends Cartan subgroups of $G$ onto Cartan subgroups of $i_G(G)$, hence $$H \subseteq i_G^{-1}\big(Z_{i_G(G)}(i_G(T))\big)=i_G^{-1}i_G(Z_G(T))=Z_G(T)\ker i_G.$$ Observe that $(\ker i_G)(k)$ is finite, since $G$ is pseudo-reductive. Hence $H^{\circ}\subseteq Z_G(T)$. But we showed earlier that $Z_G(T)$ is a smooth levitation of $T'$ in $G$, so $H^{\circ}=Z_G(T)$.

 We next assume that $H'$ is a $T'$-root group of $G'$, say $H'=U_{\alpha'}$ for some $\alpha' \in \Phi(G',T')$. The aforementioned injection $\iota:\Phi(G',T') \hookrightarrow \Phi(G,T)$ factors through a bijection $\iota_0:\Phi(G',T') \to \Phi(i_G(G),i_G(T))$. Let $\alpha_0:=\iota_0(\alpha')$ and consider the associated $i_G(T)$-root group $U_{\alpha_0}$ of $i_G(G)$. Once again define $\alpha \in \Phi(G,T)$ as follows: if $\iota(\alpha')$ is non-divisible in $\Phi(G,T)$ then set $\alpha=\iota(\alpha')$, otherwise set $\alpha=\iota(\alpha')/2$. Consider the associated $T$-root group $U_{\alpha}$ of $G$. We showed earlier that $U_{\alpha}$ is a smooth levitation of $H'$ in $G$, so $U_{\alpha} \subseteq H$.

 Recall from Lemmas \ref{adjunctionprop} and \ref{adjunctionpropnew} that $R_{k'/k}(H')$ is the largest levitation of $H'$ in $R_{k'/k}(G')$. Observe that $R_{k'/k}(H')$ is $i_G(T)$-stable, since it is smooth and all of its $i_G(T)(k)$-conjugates are also smooth levitations of $H'$ in $R_{k'/k}(G')$. Combining \cite[2.3.6, 2.3.16]{CGP} tells us that $i_G(G) \cap R_{k'/k}(H') =U_{\alpha_0}$; in particular $i_G(G) \cap R_{k'/k}(H')$ is smooth. So $U_{\alpha_0}$ is the largest levitation of $H'$ in $i_G(G)$. Clearly $i_G(U_{\alpha}) \subseteq U_{\alpha_0}$; in fact equality holds by Lie algebra considerations. 
Then $$H \subseteq i_G^{-1}(U_{\alpha_0})=i_G^{-1}i_G(U_{\alpha})=U_{\alpha}\ker i_G.$$ Once again since $(\ker i_G)(k)$ is finite, we deduce that $H^{\circ}=U_{\alpha}$.

 Finally, assume that $H'$ is a parabolic subgroup of $G'$. Observe that $H'$ is generated by the $T'$-root groups $U_{\alpha'}$ for each $\alpha' \in \Phi(H',T')$ along with $T'$. Define $\alpha \in \Phi(G,T)$ as previously. Consider the subgroup $P$ of $G$ generated by each of these $T$-root groups $U_{\alpha}$ along with $Z_G(T)$. Certainly $P$ is a pseudo-parabolic subgroup of $G$, and it is a levitation of $H'$ in $G$. Hence $P$ is contained in $H$. 
Then $H$ is also a pseudo-parabolic subgroup of $G$ by \cite[3.5.8]{CGP}. In particular, $H$ is connected. If $H$ is strictly larger than $P$ then the set of non-multipliable roots of $\Phi(H,T)$ is strictly larger than that of $\Phi(P,T)$, which is in natural bijection with $\Phi(H',T')$, contradicting the fact that $H$ is a levitation of $H'$ in $G$. So $H=P$. This completes the proof of (ii).
\end{proof}

\begin{cor}
\label{thm:levitating_regsubgroupscor}
 Let $G$ be a pseudo-reductive $k$-group, where the root system of $G_{k^{\sep}}$ is reduced. Let $H'$ be a smooth subgroup of $G'$. Suppose there exists a maximal torus $T'$ of $G'$ that normalises $H'$, and suppose $T'$ levitates.  Then $H'$ has a smooth levitation $H$.  If $H'$ is a torus then there exists a unique such $H$ that is a torus.
\end{cor}

\begin{proof}
Combine Theorem~\ref{thm:levitating_regsubgroups}(i) with Lemma~\ref{thm:levitating_tori}.
\end{proof} 

 In the setting of Theorem \ref{thm:levitating_regsubgroups} and Corollary \ref{thm:levitating_regsubgroupscor}, it is possible that $H'$ is connected whilst the largest smooth levitation of $H'$ in $G$ is not connected. Consider the following example.

\begin{example}\label{mini} Following \cite[9.1.10]{CGP}, over a suitably chosen imperfect field $k$ of characteristic $2$ one can construct a pseudo-simple $k$-group $G$ of type $A_1$ with $\ker i_G \cong \Z/2\Z$. Let $T$ be a maximal torus of $G$, and let $U$ be a $T$-root group of $G$. Denote $T':=\pi'(T_{k'})$, and consider the $T'$-root group $\pi'(U_{k'})=:H'$ of $G'$. Then $U\ker i_G$ is the largest levitation of $H'$ in $G$; it is clearly smooth. Observe that $\ker i_G \subsetneq U$ since $\ker i_G$ is central in $G$, so $U\ker i_G$ is not connected.
\end{example}

\subsection{Non-levitating subgroups}\label{woundgroups} 

 In this subsection we present two examples of non-levitating subgroups, both of which are variants of the same underlying construction. In each case the root system is reduced and the hypotheses of Theorem~\ref{thm:levitating_crit} are satisfied.
Both examples are in characteristic 2; however one could construct analogues for arbitrary characteristic $p>0$.

 In the first example we construct a non-smooth subgroup $H'$ of a reductive $k'$-group $G'$ such that $H'$ does not levitate in $G=R_{k'/k}(G')$.

\begin{example}\label{ex:nonsmooth_nonlev} Let $k$ be an imperfect field of characteristic 2 and let $k'=k(a)$, where $a^2\in k$ and $a\not\in k$. Let $K=k'(\sqrt{a})$. Let $A'$ be a $k'$-algebra. Consider the free $A'$-module $A' \otimes_{k'} K$, and its $A'$-basis $\{1 \otimes 1,1 \otimes \hspace{-0.4mm}\sqrt{a}\}$. The left regular representation of $(A' \otimes_{k'} K)^{\times}$ on $A' \otimes_{k'} K$ induces an embedding $$R_{K/k'}(\mathbb{G}_m)(A'):=(A' \otimes_{k'} K)^{\times}=\bigg\{\begin{pmatrix} x & ya \\ y & x \end{pmatrix} \bigg|\hspace{0.5mm} x,y \in A'; ~x^2 + y^2\hspace{-0.4mm}a \neq 0\bigg\} \subset \GL_2(A').$$ 
Below we write $(x,y)$ as shorthand for $\begin{pmatrix} x & ya \\ y & x \end{pmatrix}$.  This subfunctor $R_{K/k'}(\mathbb{G}_m)$ of the $k'$-group $\GL_2$ is representable, and contains the center $Z(\GL_2) \cong \mathbb{G}_m$. Set $H':=R_{K/k'}(\Gm)\cap\SL_2$. A simple calculation shows that $H'=R_{K/k'}(\mu_2)$. In particular, $\dim H' =1$ and $H'$ is not smooth.

 Now consider the map $q_{H'}:R_{k'/k}(H')_{k'} \to H'$ defined in $(\ref{counitadjunctioncomp})$. We claim that its image on $\overline{k}$-points is trivial, where $\overline{k}$ is an algebraic closure of $k$. 

 We identify $R_{k'/k}(H')(\ovl{k})$ with $H'(\ovl{k}\otimes_ k k')$, and by \cite[A.5.7]{CGP} we can identify $q_{H'}(\ovl{k})$ with the map $H'(m)\colon H'(\ovl{k}\otimes_ k k')\to H'(\ovl{k})$ induced by the multiplication map $m\colon \ovl{k}\otimes_ k k'\to \ovl{k}$, $c'\otimes d\mapsto c'd$. So let $(x,y)\in H'(\ovl{k}\otimes_ k k')$; then $x,y\in \ovl{k}\otimes_ k k'$ and $x^2+ (1\otimes a)y^2= 1$. We can write $(x,y)= (s\otimes 1+ t\otimes a, u\otimes 1+ v\otimes a)$ for some $s,t,u,v\in \ovl{k}$. Then we have \begin{align*} 1\otimes 1 &= (s\otimes 1+ t\otimes a)^2+ (1\otimes a)(u\otimes 1+ v\otimes a)^2 \\& = s^2\otimes 1+ t^2\otimes a^2+ u^2\otimes a+ v^2\otimes a^3 \\ & = (s^2+ t^2a^2)\otimes 1+ (u^2+ v^2a^2)\otimes a. \end{align*} Since $1$ and $a$ are linearly independent over $k$, we deduce that $s^2+ t^2a^2= 1$ and $u^2+ v^2a^2= 0$, so $s+ ta= 1$ and $u+ va= 0$. Hence $$H'(m)(x,y)= H'(m)(s\otimes 1+ t\otimes a, u\otimes 1+ v\otimes a)= (s+ ta, u+ va)= (1,0).$$ The claim follows.
 
 Now let $G'$ be any reductive $k'$-group that contains $H'$, 
and let $G:=R_{k'/k}(G')$. Let $H$ be any -- not necessarily smooth -- subgroup of $G$ such that $q_{G'}(H_{k'})\subseteq H'$. Then $H\subseteq R_{k'/k}(H')$ by Lemma \ref{adjunctionpropnew}, so
 $$q_{G'}(H_{k'})\subseteq q_{G'}(R_{k'/k}(H')_{k'})= q_{H'}(R_{k'/k}(H')_{k'}).$$ But $q_{H'}\colon R_{k'/k}(H')_{k'}\to H'$ is not surjective by the claim, as $H'$ is positive-dimensional. This shows that $H'$ does not levitate in $G$.
\end{example}

We next give an example of a smooth wound unipotent subgroup $U'$ of $G'$ which does not levitate in $G$.

\begin{example}\label{ex:smooth_nonlev} Let $k$, $k'$, $K$, $\ovl{k}$ be as in Example \ref{ex:nonsmooth_nonlev}, regard $R_{K/k'}(\Gm)$ as a $k'$-subgroup of $\GL_2$, and set $H':=R_{K/k'}(\Gm) \cap \SL_2$. As we showed in Example \ref{ex:nonsmooth_nonlev}, $q_{H'}$ is not surjective.

 Let $U'$ be the image of $R_{K/k'}(\Gm)$ under the canonical projection $\GL_2 \to \PGL_2$. Observe that $U'$ is smooth and wound unipotent. Let $\phi'\colon \SL_2\to \PGL_2$ be the canonical projection, and note that $(\phi')^{-1}(U')=H'$. Set $G':=\PGL_2$ as a $k'$-group, and $G:=\mathscr{D}(R_{k'/k}(G'))$. 
 
 Suppose $U$ is a levitation of $U'$ in $G$. Since $\ker i_G$ is trivial, we have a canonical inclusion $U\subseteq R_{k'/k}(U')$ by Lemma \ref{adjunctionpropnew}. The Weil restriction functor is continuous so it preserves preimages, and hence $$R_{k'/k}(\phi')^{-1}(U)\subseteq R_{k'/k}(\phi')^{-1}(R_{k'/k}(U')) = R_{k'/k}(H').$$

By functoriality of the counit (Diagram (\ref{functorialityofqG'})), we have that $\smash{\phi' \circ q_{H'}=q_{U'} \circ R_{k'/k}(\phi')_{k'}}$. Now $\phi'(\ovl{k})$ gives a bijection from $H'(\ovl{k})$ to $U'(\ovl{k})$, 
and $q_{U'}(\ovl{k})$ gives a surjection from $U(\ovl{k})$ to $U'(\ovl{k})$ by hypothesis. Moreover, the image of $R_{k'/k}(\phi')$ is $G$ by \cite[1.3.4]{CGP}. It follows that $q_{H'}$ is surjective, which is a contradiction. We conclude that $U'$ does not levitate in $G$, smoothly or otherwise. On the other hand, every torus of $G'$ levitates in $G$ by Theorem \ref{thm:levitating_crit}.
\end{example} 

\section{Maximal subgroups}
\label{sec:maximal}

 We turn to a study of maximal smooth subgroups of pseudo-reductive groups, using the results from the previous sections as a tool.  There is a long history of studying maximal subgroups of a simple linear algebraic group over an algebraically closed field; in this context, ``subgroup'' is taken to mean ``smooth subgroup''.  It is known that $G$ has only finitely many conjugacy classes of maximal smooth subgroups of positive dimension 
and there are explicit lists of these due to various authors including Borel-de Siebenthal \cite{BD}, Dynkin \cite{Dy,Dy1}, Seitz \cite{Se, Se1}, Liebeck-Seitz \cite{LS,LS1,LS2} and Testerman \cite{Te}. This gives us a relatively good understanding of the maximal smooth subgroups of an arbitrary reductive group over an algebraically closed field. Over an imperfect field the situation is less clear -- in Example \ref{ex:lots_of_wounds} we demonstrate that a simple $k$-group can admit infinitely many isomorphism classes of positive-dimensional maximal smooth subgroups. 

The smoothness requirement is indeed essential, as the following remark shows. 

\begin{rem}\label{rem:nonsmooth_maximal} 
The notion of maximality is not very well-behaved if we allow non-smooth subgroups. For instance, let $G$ be a split semisimple  
 group defined over the finite field ${\mathbb F}_p$ and let $H$ be any proper subgroup of $G$. For $r\geq 1$, let $G_r$ be the $r$th Frobenius kernel of $G$, a normal subgroup of $G$. Since each $G_r$ is infinitesimal, $\dim HG_r = \dim H < \dim G$. Let $U$ be a root group of $G$ that is not contained in $H$ (one must exist, since $G$ is generated by its root groups). Let $n$ be the smallest integer such that the finite group $U \cap H$ does not contain a copy of $\alpha_{p^n}$. Since $U \cap G_r \cong \alpha_{p^r}$ for all $r \geq 1$, we have an ascending chain of proper subgroups of $G$ $$HG_n\subseteq HG_{n+1}\subseteq \cdots$$ which never becomes stationary. On the other hand, if $G$ is a pseudo-simple group then any positive-dimensional smooth subgroup $H$ of $G$ lies in a maximal smooth subgroup.  To see this, let $M$ be a smooth proper subgroup of $G$ such that $M\supseteq H$ and $M$ has maximal dimension.  If $K$ is a proper smooth subgroup of $G$ and $M\subseteq K$ then $K^0= M^0$, so $K\subseteq N_G(M^0)$, so $K\subseteq N_G(M^0)^{\Sm}$.  Note that $N_G(M^0)^{\Sm}$ is proper as $\dim(M)> 0$ and $G$ is simple.  It follows that $N_G(M^0)^{\Sm}$ is a maximal smooth subgroup.
 
 Hence we restrict ourselves to smooth subgroups in the study of maximality.
\end{rem}
 
\begin{example}\label{ex:lots_of_wounds} Let $k$ be an imperfect field of characteristic $2$ and let $k'=k(\hspace{-0.4mm}\sqrt{a})$ for some $a \in k \setminus k^2$. As in Example \ref{ex:nonsmooth_nonlev}, consider the canonical embedding of $R_{k'/k}(\Gm)$ in $\GL_2$. Taking the quotient of $R_{k'/k}(\mathbb{G}_m)$ by the center $Z(\GL_2) \cong \mathbb{G}_m$ gives us an embedding of $\smash{U:=R_{k'/k}(\mathbb{G}_m)/\mathbb{G}_m}$ into $\PGL_2$. By choosing infinitely many elements $a \in k \setminus k^2$ that are pairwise distinct modulo $k^2$, we can construct infinitely many such subgroups $U$;  they are pairwise non-isomorphic because the minimal fields of definition of their geometric unipotent radicals are all different. All such subgroups $U$ constructed in this manner are smooth, connected, wound unipotent, $1$-dimensional, and are maximal smooth in $G$.
\end{example}


\begin{example}\label{ex:non_maximal_levitation} Let $k'/k$ be a non-trivial purely inseparable finite field extension.
%
%
Assume that $k$ has characteristic $2$. Let $G'=\mathrm{SO}_7$ and let $H'$ be an irreducibly embedded copy of $G_2$ in $G'$; note that $H'$ is maximal smooth in $G'$. 
Once again let $\smash{G:=R_{k'/k}(G')}$ and $\smash{H:=R_{k'/k}(H')}$. Observe that $H$ is perfect since $H'$ is simply connected \cite[A.7.11]{CGP}.
Hence $H$ is properly contained in the (smooth) derived subgroup $\mathscr{D}(G)$ of $G$. We claim that $G$ is not perfect. To see this, consider the simply connected cover $\mathrm{Spin}_7$ of $G'$ and its centre $Z \cong \mu_2$. The central quotient $R_{k'/k}(\mathrm{Spin}_7)/R_{k'/k}(Z) \to \mathscr{D}(G)$ is an isomorphism by \cite[1.3.4]{CGP}. But $\dim R_{k'/k}(Z)>0$ since $Z$ is infinitesimal, and so counting dimensions implies that $G$ is not perfect.  We conclude that $H$ is not maximal.
\end{example}

\begin{proof}[Proof of Theorem~\ref{thm:nicetheorem}] 
We first prove (i). Let $H$ be a maximal smooth subgroup of $G$. Denote $H':=\pi'(H_{k'})$. If $H'=G'$ then Theorem \ref{thm:levitating}(i) tells us that $H_{k^{\sep}}$ contains an almost Levi subgroup of $G_{k^{\sep}}$. So henceforth we can assume that $H' \neq G'$. Then, by maximality, $H$ is the largest smooth levitation of $H'$ in $G$ (which exists by Proposition \ref{largestlevitation}).

 Now let $M'$ be a smooth proper subgroup of $G'$ that contains $H'$. We claim that either of the conditions in the statement of the theorem is enough to ensure that $M'$ smoothly levitates in $G$. Given the claim, again using Proposition \ref{largestlevitation}, there exists a largest smooth levitation $M$ of $M'$ in $G$. Of course $H \subseteq M \subsetneq G$. Hence $H=M$ by maximality, so $H'=M'$. So indeed $H'$ is a maximal smooth subgroup of $G'$. It remains to prove the claim.

 Assume first that $G=R_{k'/k}(G')$. Then Lemma \ref{adjunctionprop} says that there exists a smooth levitation of $M'$ in $G$, namely $R_{k'/k}(M')$.

 Next assume that $H$ is a regular subgroup of $G$. Let $T$ be a maximal torus of $G$ that normalises $H$. By maximality, either $H$ contains $T$ or $HT=G$. Let $T'= \pi'(T_{k'})$, a maximal torus of $G'$.  If $H$ contains $T$ then $T' \subseteq H' \subseteq M'$. If $HT=G$ then $H'T'= G'$, so $H'$ must contain ${\mathscr D}(G')$, so every overgroup of $H'$ is normal in $G'$.  In both cases $M'$ is normalised by $T'$, so we can apply Theorem \ref{thm:levitating_regsubgroups}(i), which tells us that $M'$ smoothly levitates in $G$.

 We next prove (ii). For simplicity, we redefine/reuse some of the notation used in the proof of (i). Let $H'$ be a maximal smooth subgroup of $G'$. By assumption there exists at least one smooth levitation of $H'$ in $G$. Then, by Proposition \ref{largestlevitation}, there exists a largest smooth levitation $H$ of $H'$ in $G$. 

 Let $M$ be a smooth subgroup of $G$ that properly contains $H$. By maximality of $H'$ and $H$ and since $\pi'(M_{k'})$ is smooth, we have that $$H'=\pi'(H_{k'}) \subsetneq \pi'(M_{k'})=G'.$$ That is, $M$ is a smooth levitation of $G'$ in $G$. But then Theorem \ref{thm:levitating}(i) tells us that $M_{k^{\sep}}$ contains an almost Levi subgroup of $G_{k^{\sep}}$. This completes the proof.
\end{proof}

\begin{question}
 If $H$ is a smooth maximal subgroup of $G$ and $\pi'(H_{k'})$ is properly contained in $G'$, must $\pi'(H_{k'})$ be a maximal smooth subgroup of $G'$?  The answer is yes under the hypotheses of Theorem~\ref{thm:nicetheorem}, but we do not know whether this is true in general.
\end{question}

\bigskip
{\textbf {Acknowledgements}}:
Work on this paper began during a visit to the Mathematisches Forschungsinstitut Oberwolfach under the Oberwolfach Research Fellows Programme; we thank them for their support. The fourth author was supported by a postdoctoral fellowship of the Alexander von Humboldt Foundation.  We're grateful to David Stewart for comments and for pointing out a mistake in an earlier draft.  For the purpose of open access, the authors have applied a Creative Commons Attribution (CC BY) licence to any Author Accepted Manuscript version arising from this submission.

\bibliographystyle{amsalpha}

\newcommand{\etalchar}[1]{$^{#1}$}
\providecommand{\bysame}{\leavevmode\hbox to3em{\hrulefill}\thinspace}
\providecommand{\MR}{\relax\ifhmode\unskip\space\fi MR }
\providecommand{\MRhref}[2]{%
	\href{http://www.ams.org/mathscinet-getitem?mr=#1}{#2} }
\providecommand{\href}[2]{#2}

\end{document}